\documentclass[12pt,dvipdfmx]{article}
\usepackage[top=30truemm,bottom=30truemm,left=25truemm,right=25truemm]{geometry}
\usepackage{amsmath,amssymb,amsthm}
\usepackage{ascmac}
\usepackage{color}
\usepackage{hyperref}
\usepackage[english]{babel} 
\usepackage{authblk} 
\usepackage{tikz}
\usepackage{enumitem}

\theoremstyle{definition}
\newtheorem{theorem}{Theorem}[section]
\newtheorem{remark}[theorem]{Remark}
\newtheorem{lemma}[theorem]{Lemma}
\newtheorem{example}[theorem]{Example}
\newtheorem{corollary}[theorem]{Corollary}
\newtheorem{definition}[theorem]{Definition}
\newtheorem{question}[theorem]{Question}

\newtheorem*{proof of claim}{Proof of Claim}

\newcounter{claimcounter}
\newtheorem*{claim}{Claim \thetheorem.\theclaimcounter}

\DeclareMathOperator{\ATR}{\mathsf{ATR}}
\newcommand{\WKLo}{\mathbf{WKL}_0}
\newcommand{\ATRo}{\mathbf{ATR}_0}
\newcommand{\ACAo}{\mathbf{ACA}_0}
\newcommand{\RCAo}{\mathbf{RCA}_0}

\DeclareMathOperator{\TI}{\mathbf{TI}}
\DeclareMathOperator{\FP}{\mathsf{FP}}
\DeclareMathOperator{\MFP}{\mathsf{MFP}}

\DeclareMathOperator{\CA}{\mathbf{CA}}

\DeclareMathOperator{\LPP}{\mathsf{LPP}}
\DeclareMathOperator{\LFP}{\mathsf{LFP}}
\DeclareMathOperator{\LCP}{\mathsf{LCP}}
\DeclareMathOperator{\CNN}{\mathsf{C}_{\omega^\omega}}
\DeclareMathOperator{\dom}{dom}

\DeclareMathOperator{\WO}{WO}
\DeclareMathOperator{\lh}{lh}

\DeclareMathOperator{\N}{\mathbb{N}}
\DeclareMathOperator{\Hier}{\mathsf{Hier}}
\renewcommand{\P}{\mathsf{P}}
\DeclareMathOperator{\Pb}{Pb}
\DeclareMathOperator{\Q}{\mathsf{Q}}

\renewcommand{\L}{\mathcal{L}}
\newcommand{\rfn}{\mathrm{rfn}}
\newcommand{\GN}[1]{\ulcorner #1 \urcorner}
\newcommand{\HYP}{\mathrm{HYP}}
\newcommand{\M}{\mathcal{M}}
\renewcommand{\phi}{\varphi}
\newcommand{\myhyphen}{\mathchar`-}
\newcommand{\KB}{\mathrm{KB}}
\newcommand{\even}{\mathrm{even}}
\newcommand{\odd}{\mathrm{odd}}

\newcommand{\qedclaim}{
  \renewcommand{\qedsymbol}{$\blacksquare$}
  \qed
  \renewcommand{\qedsymbol}{$\box$}
}

\author[1]{YUDAI SUZUKI}
\author[2]{KEITA YOKOYAMA}
\affil[1,2]{Mathematical Institute, Tohoku University, Sendai, Japan}
\affil[1]{yudai.suzuki.q1@dc.tohoku.ac.jp}
\affil[2]{keita.yokoyama.c2@tohoku.ac.jp}

\title{Searching problems above \\arithmetical transfinite recursion}
\date{\today}

\begin{document}

\maketitle

\begin{abstract}
We investigate some Weihrauch problems between $\mathsf{ATR}_2$ and $\mathsf{C}_{\omega^\omega}$.
We show that the fixed point theorem for monotone operators on the Cantor space (a weaker version of the Knaster-Tarski theorem) is not Weihrauch reducible to $\mathsf{ATR}_2$.
Furthermore, we introduce the $\omega$-model reflection $\mathsf{ATR}_2^{\mathrm{rfn}}$ of $\mathsf{ATR}_2$ and show that it is an upper bound for problems provable from the axiomatic system $\mathbf{ATR}_0$ which are of the form $\forall X(\theta(X) \to \exists Y \eta(X,Y))$ with arithmetical formulas $\theta,\eta$. We also show that Weihrauch degrees of relativized least fixed point theorems for monotone operators on the Cantor space form a linear hierarchy between $\mathsf{ATR}_2^\mathrm{rfn}$ and $\mathsf{C}_{\omega^\omega}$.
\end{abstract}

\section{Introduction}

In this paper, we investigate the structure of Weihrauch degrees above Arithmetical Transfinite Recursion
by comparison with reverse mathematics.

The axiomatic system of the Arithmetical Transfinite Recursion ($\ATRo$) is known as one of the \textit{big five} systems of second-order arithmetic for the program of reverse mathematics, whose aim is to classify mathematical theorems by means of axiomatic strength.
$\ATRo$ is known as the second strongest system among the big five systems.
It is known that there is a large gap between $\ATRo$ and the strongest system $\Pi^1_1\mathchar`-\CA_0$.
This difference may be known as the difference of \textit{predicative} and \textit{impredicative}.
In \cite{MR3145191}, Towsner investigated a hierarchy between $\ATRo$ and $\Pi^1_1\mathchar`-\CA_0$.

By interpreting results of  reverse mathematics into computability theory, we can give a more precise classification of problems. In this paper, we focus on the classification of searching problems (also called Weihrauch problems), that is,  partial functions from $\mathcal{P}(\omega)$ to $\mathcal{P}(\mathcal{P}(\omega))$.
More precisely, we identify a mathematical problem $\P$ of the form $\forall X(\theta(X) \to \exists Y  \eta(X,Y))$ as the partial function $F_{\P}$ defined by $X \mapsto \{Y:\eta(X,Y)\}$ whose domain is $\{X: \theta(X)\}$.
In this paper, we mainly consider problems given by arithmetical formulas $\theta,\eta$ in the above form.
We call them \textit{arithmetical problems}.

We will see the relationship between reverse mathematics and Weihrauch degrees. Weihrauch degrees are degrees of searching problems induced by  the Weihrauch reduction according to the uniform constructivity of solutions.
The relationship between Weihrauch degrees and problems provable from the first three systems of \textit{big five}, $\RCAo, \mathbf{WKL}_0,\ACAo$ is widely studied, e.g. in \cite{Brattka_Gherardi, Gherardi_and_Marcone}.
For general settings, see also \cite{Brattka_Gherardi_Pauly, Dzhafarov_and_Mummert}.
Furthermore, some recent studies of this area focus on problems stronger or equal to $\ATRo$ \cite{MR4231614,arXiv1905.06868}.
We will introduce some new examples of such problems and investigate their Weihrauch degrees.

Remember that Arithmetical Transfinite Recursion states that \textit{for any set $X$ and well-ordering $W$, there exists a jump-hierarchy along $W$ starting with $X$}.
From the perspective of Weihrauch degrees, there are two searching problems which are analogues of $\ATRo$.
The weaker one is  $\ATR$, introduced in \cite{MR4231614}, and the stronger one is $\ATR_2$, introduced in \cite{arXiv1905.06868}.
The inputs of $\ATR$ are well-orderings, and it outputs the jump-hierarchy for a given input.
The inputs of $\ATR_2$ are any linear-orderings, and it outputs a descending sequence of a given ordering or a jump-hierarchy along it.
That is, $\ATR_2$ reflects the fact that any linear-ordering is either well-founded or ill-founded.
This use of the law of excluded middle is essential for showing that $\ATR_2$ is stronger than $\ATR$.
Indeed, there is a stronger form of the law of excluded middle used in the study of reverse mathematics within $\ATRo$
called the pseudo-hierarchy method.
More precisely, the pseudo-hierarchy method states that
\textit{for any $\Sigma^1_1$ formula $\psi(\preceq,X)$ and a set $A$ such that any well-ordering $W$ satisfies $\psi(W,A)$, there is an ill-founded linear order $L$ which satisfies $\psi(L,A)$ and has a jump hierarchy along $L$ starting with $A$}.
In d'Auriac \cite{PaulElliot}, it is shown that an analogue of the pseudo-hierarchy method is stronger than any arithmetical problem in the sense that choice on Baire space ($\CNN$) is Weihrauch reducible to it.
However, since $\ATRo$ is axiomatizable of the form $\forall X(\theta(X) \to \exists Y \eta(X,Y))$ with arithmetical formulas $\theta,\eta$ and $\CNN$ is stronger than any problem of this form, this representation of the pseudo-hierarchy method is too strong.
In this paper, we introduce an arithmetical problem which is
stronger than any arithmetical problem provable from $\ATRo$.
We also show that a version of the Knaster-Tarski theorem is not reducible to $\ATR_2$, although it is provable from $\ATRo$ by the pseudo-hierarchy method.

\subsection*{Structure of this paper}

Section 2 is for preliminaries. We introduce some basic notions of reverse mathematics and the Weihrauch hierarchy.

In Section 3, we introduce a Weihrauch problem called $\FP$ and study its Weihrauch degree.
It is shown that $\FP$ is provable in $\ATRo$ by the pseudo-hierarchy method \cite{MR1412508,bartschi2021atro}.
We prove that $\FP$ is not reducible to $\ATR_2$.
This implies that a strong Weihrauch problem is needed to capture provability from $\ATRo$.

In Section 4, we introduce the $\omega$-model reflection of Weihrauch problems and study its properties.
We show that the $\omega$-model reflection of an arithmetically representable problem is a bit stronger than the original one, and
the $\omega$-model reflection $\ATR_2^{\rfn}$ of $\ATR_2$ is sufficiently strong to reduce any arithmetical problem provable from
a theory $T \subseteq \{\sigma : \text{any $\omega$-model of $\ATRo$ satisfies $\sigma$} \}$.
In Section 5 and 6, we study the Weihrauch hierarchy by relativized leftmost path principles $(\LPP)$ and show that each $\LPP$ is a generalization of $\FP$.
In \cite{MR3145191}, it is shown that $\LPP$s make a hierarchy between $\ATRo$ and $\Pi^1_1\mathchar`-\CA_0$ in the context of reverse mathematics.
In this paper, we prove that $\LPP$s make a hierarchy between $\ATR_2^{\rfn}$ and $\CNN$. Moreover, we show that
there are large differences between each $\LPP$.

\section{Preliminaries}
\subsection{Trees}
In this subsection, we introduce the notion of trees and some related things.

\begin{definition}
  Let $X$ be a subset of $\omega$.
  We write $X^{<\omega}$ for the set of finite sequences of elements of $X$.

  Let $T$ be a subset of $X^{<\omega}$. We say $T$ is a tree on $X$ if
  $T$ is closed under taking initial segments.

  Let $T \subseteq X^{<\omega}$ be a tree.
  A function $f \in X^{\omega}$ is called a path of $T$ if $(\forall n)(f[n] \in T)$. Here, $f[n]$ denotes the initial segment of $f$ with length = $n$.
  The set of all paths of $T$ is denoted by $[T]$. We say $T$ is ill-founded if $[T] \neq \varnothing$.
\end{definition}

\begin{definition}
  Let $T$ be a tree.
  Define the Kleene-Brouwer ordering $\KB(T)$ on $T$ by
  \begin{align*}
    \sigma \leq_{\KB(T)} \tau \Leftrightarrow \tau \preceq \sigma \lor \exists j < \min\{|\sigma|,|\tau|\} (\sigma[j] = \tau[j] \land \sigma(j) < \tau(j)).
  \end{align*}
\end{definition}

For any tree $T$, $\KB(T)$ is a linear order. In addition, the following holds.
\begin{lemma} [{\cite[V.1.3]{MR2517689}}]
  For any tree $T$, $\KB(T)$ is ill-founded if and only if $T$ is ill-founded.
\end{lemma}

\subsection{Weihrauch degrees}
In this subsection, we define Weihrauch problems and Weihrauch reductions.
Then, we introduce some specific Weihrauch problems.

We often regard Weihrauch problems as representations of mathematical theorems.
From this viewpoint, Weihrauch reductions are used to compare the complexity of mathematical theorems according to computability.

\begin{definition}[Weihrauch problems]
  A partial function $\P : \subseteq \mathcal{P}(\omega) \to \mathcal{P}(\mathcal{P}(\omega))$ is called a Weihrauch problem, or a problem simply.

  Let $\P$ be a problem. We say a set $X \in \mathcal{P}(\omega)$ is an input for $\P$ if $X$ is in the domain of $\P$.
  For an input $X$, we say $Y$ is an output of $\P(X)$ if $Y \in \P(X)$.
\end{definition}

\begin{remark} \label{dom is nonempty}
  In this paper, for simplicity, we assume that any problem $\P$ has nonempty domain.
\end{remark}

\begin{remark}\label{Rmk code of seq}
  We give two remarks for our definition of Weihrauch problems.
  \begin{enumerate}
    \item In the usual definition, Weihrauch problems are defined as partial functions from $\omega^\omega$ to $\mathcal{P}(\omega^\omega)$ but this difference is not essential.
     More precisely, for any problem $\P:\subseteq \omega^\omega \to \mathcal{P}(\omega^\omega)$, there exists a problem
    $\Q:\subseteq \mathcal{P}(\omega) \to \mathcal{P}(\mathcal{P}(\omega))$ such that $\P$ is Weihrauch equivalent to $\Q$.
    \item We note that there are canonical computable bijections between $\mathcal{P}(\omega)$ and $\mathcal{P}(\omega^n)$. Similarly, there are canonical bijections between $\mathcal{P}(\omega)$ and $(\mathcal{P}(\omega))^n$, or $\mathcal{P}(\omega)$ and  $(\mathcal{P}(\omega))^{\omega}$.
  So we can identify a set $A$ as a subset of $\omega^n$ or a finite sequence of sets $\langle A_i \rangle_{i < n}$ or an infinite sequence $\langle A_i \rangle_{i \in \omega}$ and vice versa. Therefore, in this paper, we do not distinguish them. For example, if $A$ is a set, then $A_i$ denotes the $i$-th element of the sequence corresponding to $A$.
  In this manner, we can consider problems whose inputs or outputs are not a single set.
    \end{enumerate}
\end{remark}

\begin{definition}[Weihrauch reductions]
  Let $\P,\Q$ be problems. We say $\P$ is Weihrauch reducible to $\Q$ (we write $\P \leq_W \Q$) if there are computable functionals $\Phi,\Psi$ such that
  \begin{align*}
    (\forall X \in \dom(\P)) (\Phi(X)\mathord{\downarrow} \in \dom(\Q) \land \forall Y \in \Q(\Phi(X)) \Psi(X, Y)\mathord{\downarrow} \in \P(X)).
  \end{align*}
  Since $\leq_W$ is a pre-order, it induces a degree structure for problems, which we call Weihrauch degrees. If $\P,\Q$ have the same degree, then we write $\P \equiv_W \Q$ and say $\P$ is Weihrauch equivalent to $\Q$.

  We also define arithmetical Weihrauch reduction $\leq_W^a$ by replacing computable functionals $\Phi,\Psi$ with arithmetically definable functionals in the definition of Weihrauch reductions.
\end{definition}

Next, we see the notion of a (pseudo) jump hierarchy.

\begin{definition}\label{Definitions of ATR etc}
  Let $L \subseteq \omega^2$ be a linear order.
  We write $|L|$ for the field $\{i \in \omega : (i,i) \in L\}$ of $L$, $\min_L$ for the minimum element of $L$, $i \leq_L j, i <_L j$ for $(i,j) \in L$ and $(i,j) \in L \land i \neq j$ respectively.

  Let $A$ be a set. We say a set $X = \langle X_i \rangle_{i \in |L|}$ is a jump hierarchy along $L$ starting with $A$ (or jump hierarchy of $(L,A)$ for short) if for any $i \in |L|$,
  \begin{align*}
    X_i = \begin{cases}
    A \text{ if } i = \min_L, \\
    (X_{<_L i})' \text{ otherwise}.
  \end{cases}
  \end{align*}
  Here, $B'$ denotes the Turing jump of $B$, and $X_{<_L i}$ denotes the sequence $\langle X_j \rangle_{j <_L i}$.
  We remark that there exists an arithmetical formula (in the sense of second-order arithmetic) $\Hier(L,A,X)$ stating that $X$ is a jump hierarchy along $L$ starting with $A$. We also note that if $L$ is a well-order, then there exists exactly one $X$ satisfying $\Hier(L,A,X)$.
\end{definition}

\begin{definition}
  We define problems $\ATR,\ATR_2$ and $\CNN$ as follows.
  \begin{itembox}[l]{$\ATR$}
  \begin{description}
    \item[ \sf Input] A well-order $W$ and a set $A$.
    \item[ \sf Output] The jump hierarchy of $(W,A)$.
  \end{description}
  \end{itembox}

  \begin{itembox}[l]{$\ATR_2$}
  \begin{description}
    \item[ \sf Input] A linear order $L$ and a set $A$.
    \item[ \sf Output] A jump hierarchy of $(L,A)$ or a descending sequence of $L$.
  \end{description}
  \end{itembox}

  \begin{itembox}[l]{$\CNN$}
  \begin{description}
    \item[ \sf Input] An ill-founded tree $T \subseteq \omega^{<\omega}$.
    \item[ \sf Output] A path of $T$.
  \end{description}
\end{itembox}
\end{definition}

\begin{definition}[Products]
  Let $\P,\Q$ be problems.
  We define the parallel product $\P \times \Q$ of $\P$ and $\Q$ as follows.
  \begin{align*}
    &\dom (\P \times \Q) = \dom \P \times \dom \Q, \\
    &(\P \times \Q )(X_{0},X_{1}) = \P(X_0) \times \Q(X_1).
  \end{align*}
  We also define the infinite parallel product $\widehat{\P}$ of $\P$ as follows.
  \begin{align*}
    &\dom (\widehat{\P}) = \prod_{i \in \omega} \dom \P, \\
    &\widehat{\P}(\langle X_i \rangle_i) = \prod_{i} \P(X_i).
  \end{align*}
\end{definition}
It is immediate that $\P \leq_W \P \times \P \leq_W \widehat{\P}$.
On the other hand, the converses in these relations do not hold in general. We say a problem $\P$ is parallelizable if
$\P \equiv_W \widehat{\P}$.
It is well-known that $\ATR$ and $\CNN$ are parallelizable.
We will see that $\ATR_2$ is not parallelizable (Theorem \ref{ATR_2 is not parallelizable}).

\subsection{Second-order arithmetic and reverse mathematics}
In this subsection, we introduce some specific axiom systems of second-order arithmetic.
After that, we give some existing results in reverse mathematics on some variants of the Knaster-Tarski fixed point theorem.

\subsubsection{Second-order arithmetic and $\omega$-models}
  We write $\L_2$ for the language of second-order arithmetic.
  Thus $\L_2 = \{0,1,+,\cdot,<,\in\}$.

We adopt the Henkin semantics for second-order arithmetic, that is,
for any $\L_2$ structure $M = (\N^M,S^M) = (\N^M,S^M,0^M,1^M,+^M,\cdot^<,<^M)$, $S^M$ is a subset of $\mathcal{P}(\N^M)$, and the symbol $\in$ is interpreted as the standard membership relation.
We call $\N^M$ the \textit{first-order part} of $M$ and $S^M$ the \textit{second-order part} of $M$.

\begin{remark}
  We use the character $\omega$ for the set of standard natural numbers, and $\mathbb{N}$ for the first-order part of an $\L_2$ structure.
\end{remark}

\begin{definition}
  We give some specific axiom systems of second-order arithmetic.
  For the details, see \cite{MR2517689}.
  \begin{itemize}
    \item $\RCAo$ consists of the axioms for discrete ordered semi-ring, $\Sigma^0_1\mathchar`-$induction and $\Delta^0_1$-comprehension.
    \item $\ACAo$ consists of $\RCAo$ and $\forall X \exists Y (Y = X')$.
    \item $\ATRo$ consists of $\RCAo$ and $\forall W,A(\WO(W) \to \exists X \Hier(W,A,X))$.
  \end{itemize}
  Here, $\WO(W)$ is the formula stating that `$W$ is a well-order'.
\end{definition}

In connection to Weihrauch degrees and reverse mathematics, we mostly focus on models of second-order arithmetic whose first-order part is $\omega$.

\begin{definition}[$\omega$-models]
  An $\L_2$ structure $M = (\N^M,S^M)$ is called an $\omega$-model if the first-order part $\N^M$ is equal to $\omega$.
  Then, any class $S \subseteq \mathcal{P}(\omega)$ is identified as an $\omega$-model $(\omega,S)$.
  We call the $\omega$-model $(\omega,\mathcal{P}(\omega))$ \textit{the intended model}.

  For an $\L_2$ theory $\mathbf{T}$, we say $S \subseteq \mathcal{P}(\omega)$ is an $\omega$-model of $\mathbf{T}$ if $(\omega,S) \models \mathbf{T}$. If it is clear from the context, we write $S \models \mathbf{T}$ instead of $(\omega,S) \models \mathbf{T}$.
\end{definition}

\begin{lemma}[Absoluteness for arithmetical sentences]\label{absoluteness for arithmetical sentences}
  Let $S$ be an $\omega$-model. Then for any arithmetical sentence $\theta$ with parameters from $S$,
  $S \models \theta$ is equivalent to $\mathcal{P}(\omega) \models \theta$, that is, $\theta$ is true in the usual mathematics.

  Moreover, any $\Pi^1_1$ sentence true in the intended model is also true in $S$, and any $\Sigma^1_1$ sentence true in $S$ is also true in the intended model.
\end{lemma}
\begin{proof}
  Immediate from induction on the construction of $\theta$.
\end{proof}

\begin{remark}\label{ACA jump ideal}
We note that $S \subseteq \mathcal{P}(\omega)$ is an $\omega$-model of $\ACAo$ if and only if it is a Turing ideal closed under Turing jumps. This is immediate from the previous lemma (see also \cite{MR2517689}).
\end{remark}

\subsubsection{Reverse mathematics of variants of the Knaster-Tarski theorem}
In this subsection, we introduce some problems whose Weihrauch degrees will be studied in this paper.
The following theorem is known as the Knaster-Tarski theorem.

\begin{definition}
  An operator $\Gamma : \mathcal{P}(\omega) \to \mathcal{P}(\omega)$ is monotone if it is monotone increasing with respect to the inclusion order.
\end{definition}
\begin{theorem}[Knaster-Tarski]
  Any monotone operator on $\mathcal{P}(\omega)$ has a least fixed point.
\end{theorem}

\begin{theorem}[Folklore]
  Over $\RCAo$, the Knaster-Tarski theorem for arithmetically definable operators is equivalent to $\Pi^1_1\mathchar`-\CA_0$.
\end{theorem}

In Avigad \cite{MR1412508}, a weaker version of the previous theorem is studied. Furthermore, that result is extended by B\"{a}rtschi \cite{bartschi2021atro}.

\begin{definition}
  Let $\varphi(n,X)$ be a formula in the negation normal form.
  We say $\varphi$ is $X$-positive if there is no subformula of the form $t \not \in X$.
\end{definition}
We can easily prove the following lemma by induction on the construction of formulas.
\begin{lemma}
  Let $\varphi(n,X)$ be an $X$-positive formula.
  Then $\RCAo$ proves
  \begin{align*}
    (\forall X,Y)[X \subseteq Y \to (\forall n) (\varphi(n,X) \to \varphi(n,Y))].
  \end{align*}
\end{lemma}
This lemma implies that if we identify a formula $\varphi(n,X)$ with the operator $X \mapsto \{n: \varphi(n,X)\}$, then any $X$-positive formula defines a monotone operator.

\begin{theorem}[Avigad \cite{MR1412508}, B\"{a}rtschi \cite{bartschi2021atro}]\label{atr fp equiv rm}
  The following assertions are equivalent over $\RCAo$:
  \begin{enumerate}
    \item $\ATRo$,
    \item the existence of a fixed point for positive $\Sigma^0_2$ operators on $\mathcal{P}(\mathbb{N})$,
    \item the existence of a fixed point for monotone $\Delta^1_1$ operators on $\mathcal{P}(\mathbb{N})$.
  \end{enumerate}
\end{theorem}

\subsection{Correlation between Weihrauch degrees and reverse mathematics}
As we mentioned, both of Weihrauch degrees and reverse mathematics measure the complexity of mathematical theorems.
In fact, they are deeply related, and many Weihrauch problems come from reverse mathematics.

\subsubsection{Weihrauch problems defined by formulas of second-order arithmetic}
In the following, we introduce a way to translate $\L_2$ formulas to Weihrauch problems.
\begin{definition}\label{def of ATR etc}
  Let $\theta(X),\eta(X,Y)$ be $\L_2$ formulas such that the intended model $(\omega,\mathcal{P}(\omega))$ satisfies:
  \begin{itemize}
    \item there exists an $X$ such that $\theta(X)$,
    \item for any $X$ with $\theta(X)$, there exists a $Y$ such that $\eta(X,Y)$.
  \end{itemize}
  In this case, we say the pair $(\theta,\eta)$ is an $\L_2$ representation of a Weihrauch problem, and
  we define a Weihrauch problem $\Pb(\theta,\eta)$ by
  $\dom(\Pb(\theta,\eta)) = \{X : \theta(X)\}$ and $\Pb(\theta,\eta)(X) = \{Y : \eta(X,Y)\}$.
  We sometimes write $\Pb(\theta,\eta)$ as follows:
  \begin{description}
    \item[ \sf Input] $X$ such that $\theta(X)$.
    \item[ \sf Output] $Y$ such that $\eta(X,Y)$.
  \end{description}

\end{definition}

\begin{remark}\label{remark of input/output of problems}
    We may say that a problem $\P$ has an $\L_2$ representation if there are $\theta$ and $\eta$ such that $\P = \Pb(\theta,\eta)$. However, the choice of the pair of $(\theta,\eta)$ is not unique in general even if it has an $\L_2$ representation.

    In comparison with reverse mathematics, we sometimes consider the complexity of an $\L_2$ representation of a problem rather than the problem itself.
\end{remark}

\begin{definition}
  For an $\L_2$ representation $(\theta,\eta)$, we define $\L_2(\theta,\eta)$ as the $\L_2$ formula $\forall X(\theta(X) \to \exists Y \eta(X,Y))$.

  Let $\Gamma,\Gamma'$ be classes of formulas and $\theta \in \Gamma,\eta \in \Gamma'$.
  We say $(\theta,\eta)$ is a $(\Gamma,\Gamma')$-representation for $\P$ if $\P = \Pb(\theta,\eta)$ holds.
  In this case, we also say $\P$ is $(\Gamma,\Gamma')$-representable.
  If $\Gamma = \Gamma' = \Sigma^1_0$, then we say the representation $(\theta,\eta)$ is arithmetical.
\end{definition}

\begin{remark}\label{remark of L_2(P)}
  We note that $\L_2(\theta,\eta)$ is defined for an $\L_2$ representation ($\theta,\eta)$ of a problem, rather than the problem itself.
  Still, we may write $\L_2(\P)$ if the representation for $\P$ is clear from the context.
\end{remark}

Throughout this paper, we fix canonical $\L_2$ representations $(\theta,\eta)$ for $\ATR$
and $(\theta',\eta')$ for $\ATR_2$ as follows:
\begin{itemize}
  \item $\theta(W,A)$ is the $\Pi^1_1$ formula stating that $W$ is a well-order.
  \item $\eta(W,A,H)$ is $\Hier(W,A,H)$ (see Definition \ref{Definitions of ATR etc}).
  \item $\theta'(L,A)$ is the arithmetical formula stating that $L$ is a linear order.
  \item $\eta'(L,A,H)$ is the arithmetical formula stating that [$\Hier(L,A,H)$ or $H$ is an infinite descending sequence of $L$].
\end{itemize}

\begin{example}
  The following three $\L_2 $ theories are equivalent.
  \begin{enumerate}
    \item $\ATRo$,
    \item $\RCAo + \L_2(\ATR)$,
    \item $\RCAo + \L_2(\ATR_2)$.
  \end{enumerate}
\end{example}

By this translation and $\omega$-models, we can introduce another way to compare the complexity of problems.
\begin{definition}[$\omega$-model reductions]
  Let $(\theta,\eta)$ and $(\theta',\eta')$ be representations.
  We say $(\theta,\eta)$ is $\omega$-model reducible to $(\theta',\eta')$ (write $(\theta,\eta) \leq_{\omega} (\theta',\eta')$) if any $\omega$-model of $\RCAo + \L_2(\theta',\eta')$  satisfies $\L_2(\theta,\eta)$.
  We also define arithmetical $\omega$-model reductions $\leq_{\omega}^a$ by replacing $\omega$-model of $\RCAo + \L_2(\theta',\eta')$ with $\omega$-model of $\ACAo + \L_2(\theta',\eta')$.
\end{definition}

\begin{remark}\label{four reductions}
  We have introduced four kinds of reductions, $\leq_W,\leq_W^a,\leq_{\omega},\leq_{\omega}^a$.
  For $\L_2$ representations of problems, Weihrauch reduction is the strongest of these, and arithmetical $\omega$-model reduction is the weakest in the following sense.

  Let $\P$ and $\Q$ be Weihrauch problems such that $\P \leq_W \Q$.
  If there are $\L_2$ formulas $\theta_{\P},\eta_{\P},\theta_{\Q},\eta_{\Q}$ such that
  \begin{itemize}
    \item $\P = \Pb(\theta_{\P},\eta_{\P}), \Q = \Pb(\theta_{\Q},\eta_{\Q})$,
    \item $\theta_{\P},\eta_{\Q} \in \Sigma^1_1$ and $\eta_{\P},\theta_{\Q} \in \Pi^1_1$,
  \end{itemize}
  then, $(\theta_{\P},\eta_{\P}) \leq_{\omega} (\theta_{\Q},\eta_{\Q})$.
  Hence, $(\theta_{\P},\eta_{\P}) \leq^a_{\omega} (\theta_{\Q},\eta_{\Q})$ also holds.
  Similarly, $\P \leq_W^a \Q$ implies $(\theta_{\P},\eta_{\P}) \leq^a_{\omega} (\theta_{\Q},\eta_{\Q})$ for
  the representations as above.

  As a contraposition, for any problems $\P$ and $\Q$, if they have representations  $(\theta_{\P},\eta_{\P}) \in (\Sigma^1_1,\Pi^1_1)$ and $(\theta_{\Q},\eta_{\Q}) \in (\Pi^1_1,\Sigma^1_1)$ such that
  $(\theta_{\P},\eta_{\P}) \not \leq_{\omega}^a (\theta_{\Q},\eta_{\Q})$,
  then $\P \not \leq_W^a \Q$ and hence $\P \not \leq_W \Q$.
\end{remark}

\subsubsection{G\"{o}del numbers and universal formulas}
In reverse mathematics, we can consider theorem schemas as in Theorem \ref{atr fp equiv rm}.
However, we can use only subsets of $\omega$ for an input for a problem.
Thus, to represent such theorem schemas as Weihrauch problems, we code formulas by natural numbers and make decoders as formulas.
Such natural numbers are called the G\"{o}del numbers, and such decoders are called universal formulas.

\begin{theorem}[Universal formulas] (See also \cite[Definition VII.1.3.]{MR2517689})
  Let $k \in \omega, k >0$. Then there exists a universal $\Sigma^0_k$ formula. That is,
  there exists a $\Sigma^0_k$ formula $\sigma$ such that for any $\Sigma^0_k$ formula $\varphi(\vec{X},\vec{x})$,
  \begin{align*}
    \forall \vec{X},\vec{x}[\varphi(\vec{X},\vec{x}) \leftrightarrow \sigma(\GN{\varphi},\langle \vec{X} \rangle, \langle \vec{x} \rangle)]
  \end{align*}
  holds.
  Here, $\GN{\varphi}$ is the Gödel number of $\varphi$,
  $\langle \vec{X} \rangle $ is a subset of $\omega$ corresponding to $\vec{X}$ by a canonical way (see also Remark \ref{Rmk code of seq}.)
  and $\langle \vec{x} \rangle$ is a code of $\vec{x}$.
\end{theorem}

By universal formulas, we can realize a problem which admit formulas as an input.
For example, we input $\{\GN{\varphi}\}$ instead of $\varphi$.

\section{Fixed points for monotone operators}
In this section, we consider the Weihrauch degree of the existence of fixed points for monotone operators on the Cantor space.
As we have seen in Theorem \ref{atr fp equiv rm}, when we consider arithmetically definable operators, both of the fixed point theorem for positive operators and the fixed point theorem for monotone operators are equivalent to $\ATRo$.
We first show that these two kinds of fixed point theorems are also Weihrauch equivalent for $\Sigma^0_2$ definable operators. Then, we will compare the Weihrauch degrees of these fixed point theorems and $\ATR_2$, and show that they are not Weihrauch reducible to $\ATR_2$. This may mean that some arguments beyond hyperarithmetical should be needed for deducing the fixed point theorem from $\ATRo$.

For the comparison of monotone operators and positive operators,
we show that a $\Sigma^0_2$ monotone operator can be defined by a positive $\Sigma^0_2$ formula.
\begin{lemma}
  (See also \cite[Lemma VIII.2.4.]{MR2517689})
  Let $\theta(X,Y,\vec{V})$ be an $X$-positive $\Pi^0_1$ formula.
  Then, there exists an $X$-positive $\Pi^0_1$ formula $\theta'(X,\vec{V})$ such that
  $\WKLo \vdash \forall X,\vec{V} (\exists Y \theta(X,Y,\vec{V}) \leftrightarrow \theta'(X,\vec{V}))$.
  Moreover, we can effectively find $\theta'$ from $\theta$.
\end{lemma}
\begin{proof}
  Write $\theta(X,Y,\vec{V}) \equiv \forall n \rho(X,Y[n],\vec{V})$ by an $X$-positive $\Sigma^0_0$ formula $\rho$.
  Define $\theta'(X)$ by
  \begin{align*}
    \forall n \exists \sigma \in 2^{<\N} (|\sigma| = n \land \rho(X,\sigma,\vec{V})).
  \end{align*}
  Then $\theta'(X,\vec{V})$ is an $X$-positive $\Pi^0_1$ formula.

  We show that $\WKLo$ proves $\forall X,\vec{V} (\exists Y \theta(X,Y,\vec{V}) \leftrightarrow \theta'(X,\vec{V}))$.
  Assume $\WKLo$.
  Take arbitrary $X$ and $\vec{V}$.
  First, assume $\exists Y \theta(X,Y,\vec{V})$. Then, for any $n$, $\rho(X,Y[n],\vec{V})$ holds.
  Conversely, assume $\forall n  \exists \sigma \in 2^{<\N} (|\sigma| = n \land \rho(X,\sigma,\vec{V}))$.
  Define a binary tree $T$ by $T = \{\sigma \in 2^{<\N} : \rho(X,\sigma,\vec{V})\}$.
  Then, by the assumption, $T$ is infinite. Therefore, $T$ has a path by weak K\H{o}nig's lemma.
  Let $f$ be a path of $T$. Then, $\forall n \rho(X,f[n],\vec{V})$ holds.
  Therefore, this $f$ witnesses $\exists Y \theta(X,Y,\vec{V})$.
\end{proof}

\begin{theorem}
  Let $\varphi(z,X,V)$ be a $\Sigma^0_2$ formula and $A \in \mathcal{P}(\omega)$ such that
  the operator $\Gamma_{\varphi,A}$ on $\mathcal{P}(\omega)$ defined by $X \mapsto \{z: \varphi(z,X,A)\}$ is monotone.
  Then, there exists an $X$-positive $\Sigma^0_2$ formula $\psi(z,X,U)$ such that
  $\forall z,X(\varphi(z,X,A) \leftrightarrow \psi(z,X,A))$.
  Moreover, we can effectively find $\psi$ from $\varphi$.
\end{theorem}
\begin{proof}
  Let $\varphi$ and $A$ be as above.
  We first note that $\forall z,X (\varphi(z,X,A) \leftrightarrow \exists Y \subseteq X \varphi(z,Y,A))$ holds.
  Indeed, if the left-hand side holds, then we can take $Y$ as $X$.
  Conversely, assume $\varphi(z,Y,A)$ holds for some $Y \subseteq X$.
  Then, $z \in \Gamma_{\varphi,A}(Y)$.
  Now $\Gamma_{\varphi,A}(Y) \subseteq \Gamma_{\varphi,A}(X)$ by the monotonicity and hence $z \in \Gamma_{\varphi,A}(X)$. This means that $\varphi(z,X,A)$ holds.

  Therefore, it is enough to find an $X$-positive $\Sigma^0_2$ formula equivalent to $\exists Y \subseteq X \varphi(z,Y,A)$.
  Write $\varphi(z,X,V) \equiv \exists y \forall x \theta'(x,y,z,X,V)$ by a $\Sigma^0_0$ formula $\theta'$.
  Then, $\exists Y \subseteq X \varphi(z,Y,A)$ can be written as
  \begin{equation*}
    \exists Y (Y \subseteq X \land \exists y \forall x \theta'(x,y,z,Y,A)),
  \end{equation*}
  which is equivalent to
  \begin{equation*}
    \exists y \exists Y (\forall u (u \in Y \to u \in X) \land \forall x \theta'(x,y,z,Y,A)).
  \end{equation*}
  Since the part $(\forall u (u \in Y \to u \in X) \land \forall x \theta'(x,y,z,Y,A))$ is $X$-positive and $\Pi^0_1$, there exists an $X$-positive $\Pi^0_1$ formula $\theta''(y,z,X,A)$ such that
  $\exists Y(\forall u (u \in Y \to u \in X) \land \forall x \theta'(x,y,z,Y,A))$ is equivalent to $\theta''(y,z,X,A)$ by the previous lemma.
  Therefore, we have $\forall z,X(\varphi(z,X,A) \leftrightarrow \exists y \theta''(y,z,X,A))$.
  By putting $\psi(z,X,U) \equiv \exists y \theta''(y,z,X,U)$, we have the desired conditions.
\end{proof}

\begin{definition}
  Let $k \in \omega, k>0$. We define $\FP(\Sigma^0_k)$  as follows.
  \begin{itembox}[l]{$\FP(\Sigma^0_k)$}
  \begin{description}
    \item[ \sf Input] A set $A$ and a Gödel number $\GN{\varphi}$ of an $X$-positive $\Sigma^0_k$ formula $\varphi(n,X,Y)$.
    \item[ \sf Output] A fixed point of the operator $X \mapsto \{n: \varphi(n,X,A)\}$.
  \end{description}
\end{itembox}
\end{definition}

\begin{corollary}
The problem $\FP(\Sigma^0_2)$ is Weihrauch equivalent to the following problem:
\begin{itembox}[l]{$\MFP(\Sigma^0_2)$}
\begin{description}
  \item[ \sf Input] A set $A$ and a Gödel number $\GN{\varphi}$ of a $\Sigma^0_2$ formula $\varphi(n,X,Y)$ such that $X \mapsto \{n:\varphi(n,X,A)\}$ is monotone.
  \item[ \sf Output] A fixed point of the operator $X \mapsto \{n: \varphi(n,X,A)\}$.
\end{description}
\end{itembox}
\end{corollary}
\begin{proof}
  Since a positive formula defines a monotone operator, $\dom(\FP(\Sigma^0_2)) \subseteq \dom(\MFP(\Sigma^0_2))$.
  In addition, for any input for $\FP(\Sigma^0_2)$, the outputs of $\FP(\Sigma^0_2)$ and $\MFP(\Sigma^0_2)$ are the same.
  Thus $\FP(\Sigma^0_2) \leq_W \MFP(\Sigma^0_2)$ holds.

  Conversely, for any input for $\MFP(\Sigma^0_2)$, we can effectively find an input for $\FP(\Sigma^0_2)$ such that they define the same monotone operator by the above lemma. Thus, $\MFP(\Sigma^0_2) \leq_W \FP(\Sigma^0_2)$.
\end{proof}

\begin{corollary}
  The problem $\FP(\Sigma^0_2)$ is parallelizable.
\end{corollary}
\begin{proof}
  It is enough to show that $\MFP(\Sigma^0_2)$ is parallelizable.

  Let $\sigma(e,x,X,Y)$ be a fixed $\Sigma^0_2$ universal formula. Let $\langle A_i \rangle_i$ be a sequence of sets and $\langle e_i \rangle_i$ be a sequence of numbers such that
  each $\sigma(e_i,x,X,A_i)$ defines a monotone operator.
  We will find a sequence $\langle X_i \rangle_i$ such that each $X_i$ is a fixed point of the operator defined by $\sigma(e_i,x,X,A_i)$.

  We define a $\Sigma^0_2$ formula $\varphi(x,X,Y)$ by
  \begin{align*}
    \varphi(n,X,Y) \equiv (\exists i,x \leq n) (n = (i,x) \land \sigma(e_i,x,X_i,Y_i)).
  \end{align*}
  Then for each $X = \langle X_i \rangle_i$,
  $\{n: \varphi(n,X,\langle A_i \rangle_i)\} = \bigcup_i \{(i,x): \sigma(e_i,x,X_i,A_i) \}$.
  Thus, $\varphi(n,X,\langle A_i \rangle_i)$ defines a monotone operator, and for any fixed point $\langle X_i \rangle_i$
  of this operator, each $X_i$ is a fixed point of the operator defined by $\sigma(e_i,x,X,A_i)$.
\end{proof}

Next, we compare $\ATR_2$ and $\FP(\Sigma^0_2)$.

\begin{theorem}\label{atr2 is reducible to fp}
  (Essentially due to Avigad, \cite{MR1412508})
    $\ATR_2 \leq_W \FP(\Sigma^0_2)$.
\end{theorem}
\begin{proof}
We will follow the proof by Avigad.

  Let $\widetilde{L}$ be a linear order and $A$ be a set.
  We will find an output of $\ATR_2(\widetilde{L},A)$.

  First, we define a linear order $L$ as follows:
  \begin{align*}
    &|L| = |\widetilde{L}| \cup \{\infty\}, \\
    &\forall a,b \in |\widetilde{L}| (a <_L b \Leftrightarrow a<_{\widetilde{L}} b), \\
    &\forall a \in |\widetilde{L}|(a <_L \infty).
  \end{align*}
  Thus, $L$ is the liner order extending $\widetilde{L}$ with a maximum element $\infty$.

  For each formula $\varphi$ in the negation normal form, define a formula $\widehat{\varphi}(n,Z,b)$ by replacing, from $\varphi$, each subformula of the form $t \in X$ by
  \begin{align*}
    t_0 <_L b \land (t,1) \in Z
  \end{align*}
  and subformula of the form $t \not \in X$ by
  \begin{align*}
    t_0 \not <_L b \lor (t,0) \in Z.
  \end{align*}
  Here, $x_0$ denotes the first coordinate inverse of the pairing function, that is, any $x$ is written as $x = (x_0,x_1)$ for some $x_1$.
  By the same way, we also define $\widehat{\lnot \varphi}$ from $\lnot \varphi$.
  We note that for any set $Z,Y$ and $b \in |L|$,
  if $\forall c <_L b \forall l \exists ! k < 2 [((c,l),k) \in Z)]$ and $Y_{<_L b} = \{(c,l) : c <_L b \land ((c,l),1) \in Z\}$ then
  \begin{align*}
    &\widehat{\varphi}(n,Z,b) \leftrightarrow \varphi(n,Y_{<_L b}), \\
    &\widehat{\lnot \varphi}(n,Z,b) \leftrightarrow \lnot \varphi(n,Y_{<_L b})
  \end{align*}
  by induction on the construction of $\varphi$.

  In the following argument, we fix a bijection between $\omega$ and $\omega \cup \{-1,-2\}$.
  Let $\varphi$ be a $\Sigma^0_1 \land \Pi^0_1$ formula in the negation normal form such that for any set $X$ if $X = \varnothing$ then $\{n: \varphi(n,X)\} =A$ and otherwise  $\{n: \varphi(n,X)\} = X'$.
  Define a formula $\psi(n,X)$ by
  \begin{align*}
    \psi(((b,m),k),X) \equiv b \in |L| \land  ((\ast) \lor (\ast\ast) \lor (\ast \ast \ast)).
  \end{align*}
  Here, $(\ast), (\ast\ast)$ and $(\ast \ast \ast)$ are the following formulas.
  \begin{align*}
    (\ast) &\equiv m = -2 \land k= 1 \land \forall c <_L b \forall l \geq 0\, (((c,l),0) \in X \lor ((c,l),1) \in X), \\
    (\ast \ast) &\equiv m=-1 \land k= 1 \land \exists c<_L b \exists l\geq0\, (((c,l),0) \in X \land ((c,l),1) \in X)), \\
    (\ast\ast\ast) &\equiv m\geq 0 \ \land \\
     [&(k = 0 \land \widehat{\lnot \varphi}(m,X-\{((d,m),i) : d \in |L|, m <0, i < 2\},b)) \ \lor \\
      &(k=1 \land \widehat{\varphi}(m,X-\{((d,m),i) : d \in |L|,m < 0,i < 2\},b))].
  \end{align*}
  We note that $\psi$ is a $\Sigma^0_2$ $X$-positive formula.

  Let $Z$ be a fixed point of $\psi$. Then, we have
  \begin{align*}
    ((b,m),k) \in Z \leftrightarrow \psi(((b,m),k),Z).
  \end{align*}
  Define $S_0$ and $S_1$ as follows:
  \begin{align*}
    S_1 &= \{b \in |L| : ((b,-2),1) \not \in Z\}, \\
    S_2 &= \{b \in |L|: ((b,-1),1) \in Z\}.
  \end{align*}

  In the following argument, we say \textit{$Z_{<_L b}$ is a characteristic function} if
  \begin{align*}
    \forall c <_L b \forall l\geq 0 \exists !k <2\,  (((c,l),k) \in Z).
  \end{align*}
  We also define the condition \textit{$Z_{\leq_L b}$ is a characteristic function} by a similar way.
  \setcounter{claimcounter}{1}
  \begin{claim}
    For any $b \in |L|$, we have
    \begin{enumerate}
      \item if $b \in S_1 \cup S_2$, then $Z_{\leq_L b}$ is not a characteristic function,
      \item if $Z_{\leq_L c}$ is a characteristic function for all $c <_L b$, then $Z_{\leq_L b}$ is also a characteristic function,
      \item if $Z_{\leq_L b}$ is not a characteristic function, then $b \in S_1 \cup S_2$.
    \end{enumerate}
  \end{claim}
  \addtocounter{claimcounter}{1}
  \begin{proof of claim}
    (1) Let $b \in S_1 \cup S_2$. If $b \in S_1$, then $((b,-2),1) \not \in Z$.
    Thus $\lnot \psi(((b,-2),1),Z)$.
    By the definition of $\psi$, there are $c <_L b$ and $l \geq 0$ such that $((c,l),0) \not \in Z \land ((c,l),1) \not \in Z$.
    Hence, $Z_{\leq b}$ is not a characteristic function.
    Similarly, assume $b \in S_2$. Then $((b,-1),1) \in Z$ and hence $\psi(((b,-1),1),Z)$.
    By the definition of $\psi$, there are $c <_L b$ and $l \geq 0$ such that $((c,l),0)  \in Z \land ((c,l),1)  \in Z$.
    Thus, $Z_{\leq b}$ is not a characteristic function.

    (2) Assume $\forall c <_L b \,(Z_{\leq_L c} \text{ is a characteristic function})$.
    Let $Y = \{(c,m) : c <_L b \land m \geq 0 \land ((c,m),1) \in Z)\}$.
    Then, for any $n \geq 0$, we have
    \begin{align*}
       \phi(n,Y) &\leftrightarrow \widehat{\varphi}(n,Z-\{((c,m),i): c <_L b, m < 0, i < 2\},b),\\
      \lnot \phi(n,Y) &\leftrightarrow \widehat{\lnot \varphi}(n,Z-\{((c,m),i): c <_L b, m < 0, i < 2\},b).
    \end{align*}
    Thus, exactly one of $\widehat{\varphi}(n,Z-\{((c,m),i): c <_L b, m < 0, i < 2\},b)$ or $\widehat{\lnot \varphi}(n,Z-\{((c,m),i): c <_L b, m < 0, i < 2\},b)$ holds.
    Thus, by the definition of $\psi$, exactly one of $((b,n),0)$ or $((b,n),1)$ is in $Z$ for any $n \geq 0$.
    This yields $Z_{\leq_L b}$ is a characteristic function because $Z_{<_L b}$ is a characteristic function.

    (3) Assume $Z_{\leq_L b}$ is not a characteristic function.
    Then, by (2), there exists $c <_L b$ such that $Z_{<_L c}$ is not a characteristic function.
    Pick such a $c$.
    Now there are two possible cases:
    \begin{enumerate}[label = (\roman*)]
      \item $\exists d <_L c \exists l \geq 0 [((d,l),0) \not \in Z \land ((d,l),1) \not \in Z]$, or
      \item $\exists d <_L c \exists l\geq 0 [((d,l),0) \in  Z \land ((d,l),1) \in Z]$.
    \end{enumerate}
    In the former case, $\psi(((b,-2),1), Z)$ does not hold and hence $b \in S_1$.
    In the latter case, $\psi(((b,-1),1),Z)$ holds and hence $b \in S_2$.
    \qed
  \end{proof of claim}
  Now we have $S_0 \cup S_1$ has no minimal element and is upward closed.
  Moreover, $\{(b,m) : m \geq 0 \land b \in |L| \land b \not \in S_1 \cup S_2 \land ((b,m),1) \in Z\}$ is a jump hierarchy.
  Thus, if $\infty \not \in S_0 \cup S_1$, then $\{(b,m) : m \geq 0 \land b \in |L| \land b \neq \infty \land ((b,m),1) \in Z\}$ is an output of $\ATR_2(\widetilde{L},A)$. If $\infty \in S_0 \cup S_1$, then we can effectively find a descending sequence from $S_0 \cup S_1$.
  Since $S_0$ and $S_1$ are uniformly computable from $Z$, the proof is completed.
\end{proof}

We give a problem which is not reducible to $\ATR_2$ but reducible to $\ATR_2 \times \ATR_2$. This yields that $\ATR_2$ is not parallelizable.
Since $\FP(\Sigma^0_2)$ is parallelizable, we also have that $\ATR_2$ is strictly weaker than $\FP(\Sigma^0_2)$.

\begin{lemma}\label{atr2 cannot reduce}
  (Essentially due to Goh et al. \cite[Lemma 4.9.]{MR4328030})
  Let $\P$ be a problem. Assume $\P$ has a computable input $A$ with the following condition. For any output $B \in \P(A)$, there exists $Z$ such that $B = 0'' \oplus Z$ and $Z$ is not $\Delta^1_1$.
  Then $\P \not \leq_W \ATR_2$.
\end{lemma}
\begin{proof}
Let $\P$ and $A$ be as above.
Assume $\P \leq_W \ATR_2$ via $\Phi,\Psi$.
Since $A$ is an input for $\P$, $\Phi(A)$ is an input for $\ATR_2$. We note that there exist a computable set $C$ and a computable linear order $L$ such that $\Phi(A) = (L,C)$. We pick such $C$ and $L$.
If $L$ is a well-order, then the unique output $Y \in \ATR_2(L,C)$ is $\Delta^1_1$ and so is $\Psi(A,Y)$. However, $\Psi(A,Y)$ should be an output of $\P(A)$, so this is a contradiction. Hence $L$ is ill-founded.

Now for any descending sequence $f$ of $L$, $\Psi(A,f)$ is an output of $\P(A)$.
Thus we can find an oracle function $\Theta^X(\bullet)$ such that for any descending sequence $f$ of $L$, $\Theta^f(\bullet)$ is the characteristic function of $0''$.

We will define a sequence of finite sequence $\mathcal{F} = \{F_{s,n}\}_{s,n \in \omega} \leq_T L$ such that
  \begin{align*}
      &(\forall n)(\exists s_n)(\forall s > s_n)(\Theta^{F_{s,n}}(n)\mathord{\downarrow}), \\
      &(\forall n)(\exists s_n)(\forall s > s_n)(\exists f_{s,n}:\text{descending sequence of } L)(F_{s,n} \prec f_{s,n}).
  \end{align*}
For each $s$ and $n$, define $D_{s,n}$ by
  \begin{align*}
      D_{s,n} = \{F : F \text{ is a finite descending sequence of $L$} \land |F| \geq 1 \land F \leq s \land \Theta^{F}_s (n)\downarrow\}.
  \end{align*}
  Here, the second inequality $F \leq s$ means that the code of $F$ is smaller than $s$, and $\Theta^{F}_s$ denotes the $s$-step restriction of $\Theta^{F}$.
  If $D_{s,n} = \varnothing$, then define $F_{s,n}$ as the empty sequence $\langle \rangle$. Otherwise, pick $F_{s,n} \in D_{s,n}$ such that
  \begin{align*}
    (\forall F \in D_{s,n})(F(\lh(F)-1) \leq_{L} F_{s,n}(\lh(F)-1)).
  \end{align*}

  We will show that $\{F_{s,n}\}_{s,n}$ satisfies the desired conditions.
  Fix $n \in \omega$. Let $f$ be a descending sequence through $L$.
  Since $\Theta^f$ is a total function, there exists $t \in \omega$ such that $\Theta^{f[t]}_t(n)\mathord{\downarrow}$.
  Thus for any $s \geq \max\{f[t],t\}$, $f[t] \in D_{s,n}$.
  Moreover, since $f[t] \prec f$ and $f(t-1) \leq_{L} F_{s,n}(|F_{s,n}|-1)$,
  the concatenation $F_{s,n} \ast \langle f(t),f(t+1),\ldots \rangle$ is an
  infinite descending sequence of $L$.

  Define an $L$-computable function $p(s,n) = \Theta^{F_{s,n}}_s(n)$. Then, we have $(\forall n)(\lim_s p(s,n) = 0''(n))$.
  On the other hand, since $p \leq_T L \leq_T 0$, we have a contradiction that $0'' \leq_T 0'$.
\end{proof}

\begin{theorem}\label{ATR_2 is not parallelizable}
    $\ATR_2 \not \leq_W \ATR_2 \times \ATR_2$. Therefore, $\ATR_2$ is not parallelizable.
\end{theorem}
\begin{proof}
  We will show that Lemma \ref{atr2 cannot reduce} can be applied for $\ATR_2 \times \ATR_2$.

  Let $A$ be the linear order with length  = 2 and $L$ be a pseudo well-order. That is, $L$ is a computable linear order but $\ATR_2(L,\varnothing)$ has no $\Delta^1_1$ output.
  Now both of $(A,\varnothing),(L,\varnothing)$ are computable inputs for $\ATR_2$, so
  $X= ((A,\varnothing),(L,\varnothing))$ is a computable input for $\ATR_2 \times \ATR_2$.
  Then any output of $(\ATR_2 \times \ATR_2)(X)$ is of the form $(0'',Z)$ such that $Z \in \ATR_2(L,\varnothing)$,
  so we can apply Lemma \ref{atr2 cannot reduce} to $\ATR_2 \times \ATR_2$.
\end{proof}

\begin{remark}
  In \cite{arXiv1905.06868}, it is shown that K\"{o}nig's duality theorem ($\mathsf{KDT}$) is parallelizable and arithmetically Weihrauch equivalent to $\ATR_2$ (see Proposition 9.4.\,and Corollary 9.28.). Thus, $\ATR_2$ is arithmetically parallelizable while it is not parallelizable.
\end{remark}

\begin{theorem}\label{fp is not reducible to atr2}
    $\FP(\Sigma^0_2)$ is not Weihrauch reducible to $\ATR_2$.
\end{theorem}
\begin{proof}
  Since $\ATR_2 \leq_W \FP(\Sigma^0_2)$ and $\FP(\Sigma^0_2)$ is parallelizable,
  $\ATR_2 \times \ATR_2 \leq_W \FP(\Sigma^0_2)$.
  Thus, $\FP(\Sigma^0_2) \not \leq_W \ATR_2$ because $\ATR_2 \times \ATR_2 \not \leq_W \ATR_2$.
\end{proof}

\section{Omega-model reflection}
In the previous section, we have shown that neither $\FP(\Sigma^0_2)$ nor $\ATR_2 \times \ATR_2$ are Weihrauch reducible to $\ATR_2$ although $\L_2(\FP(\Sigma^0_2)), \L_2(\ATR_2 \times \ATR_2)$ and $\L_2(\ATR_2)$ for canonical representations are equivalent over $\RCAo$.
This means that $\ATR_2$ is too weak to capture the provability over $\ATRo$ even when we consider only arithmetically representable problems.
Then, what is a useful upper bound for arithmetically representable problems provable from $\ATRo$?

In general, it is known that any countable chain of Weihrauch degrees has no nontrivial supremum\cite{higuchi2010degree}.
Hence one cannot expect that there exists a supremum of all problems provable from $\ATRo$.
In this chapter, we introduce an operator the $\omega$-model reflection $\bullet^{\rfn}$ on Weihrauch degrees and show that the $\omega$-model reflection of an arithmetically representable problem $\P$ is still arithmetically representable but strong enough to capture the provability from $\L_2(\P)$.

\begin{definition}[Effective $\omega$-models]
  Let $\langle M_i \rangle_i \subseteq \mathcal{P}(\omega),f,g \in \omega^{\omega}$.
  We say a triple $(\langle M_i \rangle_i,f,g)$ is an effective $\omega$-model of $\ACAo$ if the following conditions hold.
  \begin{align*}
    &M = \langle M_i \rangle_i \text{ is lower closed under Turing reduction} \leq_T, \\
    &M_{f(i)} = (M_i)',\\
    &M_{g(i,j)} = M_i \oplus M_j.
  \end{align*}
\end{definition}
As mentioned in Remark \ref{ACA jump ideal}, an $\omega$-model $S \subseteq \mathcal{P}(\omega)$ satisfies $\ACAo$ if and only if it is lower closed under Turing reductions $\leq_T$, closed under Turing sums $\oplus$ and Turing jumps $\bullet'$.
Thus, an effective $\omega$-model of $\ACAo$ is an $\omega$-model of $\ACAo$
in which one may find jumps and sums effectively.

\begin{definition}[$\omega$-model reflections]
    Let $\P$ be a problem.
    We define the $\omega$-model reflection $\P^{\rfn}$ of $\P$ as follows.
    \begin{itembox}[l]{$\P^{\rfn}$}
    \begin{description}
        \item[\sf Input] Any set $X$.
        \item[\sf Output] A tuple $\mathcal{M} = (\langle M_i \rangle_i,f,g,e)$ such that
        \begin{align*}
            &(\langle M_i \rangle_i,f,g) \text{ is an effective $\omega$-model of $\ACAo$}, \\
            &M_e = X,\\
            &\forall i(M_i \in \dom \P \to \exists j (M_j \in \P(M_i))).
        \end{align*}
    \end{description}
  \end{itembox}
  For a set $Y$, we write $Y \in \M$ to mean $Y = M_i$ for some $i$. We also write
  $\M \models \varphi$ to mean $\langle M_i \rangle_i \models \varphi$ for an $\L_2$ sentence $\varphi$.
\end{definition}

\begin{remark}\label{If P is arithmetical, then the last condition is equiv}
  We note that if $\P$ is represented by $(\theta,\eta)$, then the last condition $\forall i(M_i \in \dom \P \to \exists j (M_j \in \P(M_i)))$ is equivalent to
  $\forall i (\theta(M_i) \to \exists j \eta(M_i,M_j))$.
  If $\theta,\eta$ are arithmetical, then this is also equivalent to $\langle M_i \rangle_i \models \L_2(\theta,\eta)$
  (see also Theorem \ref{output-of-omega-model-ref-is-model}).
\end{remark}

\subsection{Basic properties of $\omega$-model reflection}
In this subsection, we will see some basic properties of the operator  $\bullet^{\rfn}$.
First of all, we show that $\bullet^{\rfn}$ is well-defined for Weihrauch degrees.

\begin{theorem}\label{well-def of reflection}
  Let $\P,\Q$ be problems.
  \begin{enumerate}
    \item If $\P \leq_W^a \Q$, then $\Q^\rfn$ refines $\P^\rfn$.
  That is, $\dom(\Q^\rfn) = \dom(\P^\rfn)$ and for any $X \in \dom(\Q^\rfn)$, $\Q^\rfn(X) \subseteq \P^{\rfn}(X)$.
  In particular, $\P^\rfn \leq_W \Q^\rfn$.
  \item If $\P \equiv_W^a \Q$, then $\P^{\rfn}$ and $\Q^{\rfn}$ are the same function. In particular, $\P^\rfn \equiv_W \Q^\rfn$.
\end{enumerate}
\end{theorem}
\begin{proof}
  (1.) Let $\P$ and $\Q$ be problems such that $\P \leq_W^a \Q$ via arithmetical operators $\Phi,\Psi$.
  We note that $\dom(\Q^\rfn) = \dom(\P^\rfn) = \mathcal{P}(\omega)$.
  Thus, what we should show is that for any set $X$, $\Q^\rfn(X) \subseteq \P^\rfn(X)$.
  Let $\M =(\langle M_i \rangle_i,f,g,e)\in \Q^{\rfn}(X)$
  and we will show that $\forall i (M_i \in \dom \P \to \exists j (M_j \in \P(M_i)))$.

  Let $i$ be such that $M_i \in \dom \P$. Then, $\Phi(M_i) \in \dom \Q$.
  Since $\Phi$ is arithmetical, $\Phi(M_i) = M_k$ for some $k$.
  Pick such a $k$. Then, there exists $l$ such that $M_l \in \Q(M_k)$ because $\langle M_n \rangle_n$ is an output of $\Q^\rfn$. For such an $l$, $\Psi(M_i,M_l) \in \P(M_i)$ and $\Psi(M_i,M_l) \in \langle M_n \rangle_n$.

  (2.) This is immediate from (1.)
\end{proof}

The name $\omega$-model reflection comes from a principle in second-order arithmetic having the same name.
We will see some properties of this principle in the next section.
Intuitively, the second-order arithmetic version of the $\omega$-model reflection says that \textit{any set is contained by an $\omega$-model of a given problem}.
Thus, one may expect that our $\omega$-model reflection outputs an $\omega$-model of a given problem.
This is actually true for $(\Sigma^1_1,\Pi^1_1)$-representations.

\begin{theorem}\label{output-of-omega-model-ref-is-model}
    Let $(\theta,\eta)$ be a $(\Sigma^1_1,\Pi^1_1)$-representation for a problem $\P$.
    Then, for any input $A$ for $\P^{\rfn}$ and any output $\M \in \P^{\rfn}(A)$, $\M \models \L_2(\theta,\eta) + \ACAo$.
\end{theorem}
\begin{proof}
  Let $A$ be a set and $\M = (\langle M_i \rangle_i,f,g,e)$ be an output of $\P^{\rfn}(A)$. Then, $\M \models \ACAo$.
  Since $\L_2(\theta,\eta)$ is $\forall X(\theta(X) \to \exists Y \eta(X,Y))$,
  it is enough to show that $\M \models \forall X(\theta(X) \to \exists Y \eta(X,Y))$.

  Let $X \in \M$ such that $\M \models \theta(X)$. Since $\theta(X)$ is $\Sigma^1_1$, $\theta(X)$ is true. Thus, by the definition of $\P^{\rfn}$, there exists $j$ such that $\eta(X,M_j)$ holds. Since $\eta$ is $\Pi^1_1$, we have $\M \models \eta(X,M_j)$.
\end{proof}
\begin{remark}
  We note that the previous lemma is optimal.
  In fact, we can show the following.
  \begin{enumerate}
    \item There exists a $(\Pi^1_1,\Sigma^1_0)$-representable problem $\P = \Pb(\theta,\eta)$ with the following property. There exists a set $A$ and an output $M \in \P^{\rfn}(A)$ such that $M \not \models \L_2(\theta,\eta)$.
    \item There exists a $(\Sigma^1_0,\Sigma^1_1)$-representable problem $\P = \Pb(\theta,\eta)$ with the following property. There exists a set $A$ and an output $M \in \P^{\rfn}(A)$ such that $M \not \models \L_2(\theta,\eta)$.
  \end{enumerate}
  For (1.), consider $\ATR$. Then there exists an output $\mathcal{M} \in \ATR^{\rfn}(\varnothing)$ such that $\langle M_i \rangle_i = \HYP$ but $\HYP$ is not a model of $\ATRo$, where $\HYP$ is the set of all hyperarithmetical sets.
  For (2.), consider a representation $(\theta,\eta)$ such that
  \begin{align*}
    &\theta(X) \equiv X \text{ is a linear order}, \\
    &\eta(X,Y) \equiv Y =Y \land \exists Z(\Hier(X,\varnothing,Z) \lor Z \text{ is a descending sequence of $X$}).
  \end{align*}
  In this case, any countable $\omega$-model of $\ACAo$ can be an output of $\P^{\rfn}(\varnothing)$. However, $\L_2(\theta,\eta)$ is equivalent to $\ATRo$ over $\RCAo$. Thus, the minimum $\omega$-model ARITH of $\ACAo$ does not satisfy $\L_2(\theta,\eta)$.
\end{remark}

Our main purpose of introducing $\omega$-model reflections is relating provabilities in second-order arithmetic and Weihrauch reducibilities.
Thus, we next see the relationship between $\omega$-model reflections and provabilities.

\begin{lemma}\label{P-ref and omega-model reducibility}
  Let $(\theta,\eta)$ be a $(\Pi^1_1,\Sigma^1_1)$-representation and $(\theta',\eta')$ be a $(\Sigma^1_1,\Pi^1_1)$-representation such that $(\theta,\eta) \leq_{\omega}^a (\theta',\eta')$.
  Put $\P = \Pb(\theta,\eta)$ and $\Q = \Pb(\theta',\eta')$.
  Then $\Q^{\rfn}$ refines $\P^{\rfn}$ and hence $\P^{\rfn} \leq_W \Q^{\rfn}$ holds.
\end{lemma}
\begin{proof}
  Let $(\theta,\eta),(\theta',\eta'),\P$ and $\Q$ be as above.

  Let $A$ be a set and $\M \in \Q^{\rfn}(A)$. We will show that $\forall X \in \M (\theta(X) \to \exists Y \in \M \eta(X,Y))$ holds.
  Let $X \in \M$ such that $\theta(X)$ holds.
  Then, $\M \models \theta(X)$ because $\theta$ is $\Pi^1_1$.
  Now we have $\M \models \L_2(\theta',\eta') + \ACAo$ by Theorem \ref{output-of-omega-model-ref-is-model}.
  Thus $\M \models \L_2(\theta,\eta)$ since $\P \leq_{\omega}^a \Q$.
  Thus there exists $Y \in \M$ such that $\M \models \eta(X,Y)$.
  Since $\eta$ is $\Sigma^1_1$, $\eta(X,Y)$ holds for such a $Y$.
\end{proof}

\begin{theorem}\label{P-ref and provability}
  Let $\P = \Pb(\theta,\eta)$ and $\Q = \Pb(\theta',\eta')$ be the same as in Lemma \ref{P-ref and omega-model reducibility}.
  If $\L_2(\theta,\eta)$ is provable from a theory $T \subseteq \{\sigma : \text{any $\omega$-model of $\L_2(\theta',\eta') + \ACAo$ satisfies $\sigma$} \}$, then $\P^{\rfn} \leq_W \Q^{\rfn}$.
  In particular, if $\L_2(\theta',\eta')$ proves $\L_2(\theta,\eta)$, then $\P^{\rfn} \leq_W \Q^{\rfn}$.
\end{theorem}
\begin{proof}
  Let $T \subseteq \{\sigma : \text{any $\omega$-model of $\L_2(\theta',\eta') + \ACAo$ satisfies $\sigma$} \}$ be such that
  $T \vdash \L_2(\theta,\eta)$.

  By Lemma \ref{P-ref and omega-model reducibility}, it is enough to show that $(\theta,\eta) \leq_{\omega}^{a} (\theta',\eta')$. Let $\M$ be an $\omega$-model of $\L_2(\theta',\eta') + \ACAo$.
  Then $\M \models T$. Hence $\M \models \L_2(\theta,\eta)$. This completes the proof.
\end{proof}

\begin{lemma}\label{P-ref-strengthen}
    If $\P$ is a $(\Sigma^1_\infty,\Sigma^1_0)$-representable problem, then $\P \leq_W \P^{\rfn}$.
\end{lemma}
\begin{proof}
    Let $(\theta,\eta)$ be a $(\Sigma^1_\infty,\Sigma^0_k)$-representation for $\P$.
    To show $\P \leq_W \P^{\rfn}$, take an input $X$ for $\P$.
    Then $X$ satisfies $\theta(X)$. We will find a set $Y$ satisfying $\eta(X,Y)$ by using $\P^{\rfn}$.

    Let $\M = (\langle M_i \rangle_i,f,g,e)$ be an output of $\P^{\rfn}(X)$.
    Since $\theta(M_e)$ holds, there exists $j$ such that $\eta(M_e,M_j)$ holds.
    Moreover, for each $i$ and $j$, the question \textit{does $\eta(M_i,M_j)$ hold?} is decidable from $M_{f^{k+1}g(i,j)}$ uniformly.
    Therefore, we can effectively find an index $j$ such that $\eta(M_e,M_j)$ holds.
\end{proof}

\begin{theorem}
  Let $(\theta,\eta)$ be a $(\Pi^1_1,\Sigma^1_0)$-representation for $\P$ and $(\theta',\eta')$ be a $(\Sigma^1_1,\Pi^1_1)$-representation for $\Q$.
  If $(\theta,\eta) \leq_{\omega}^a (\theta',\eta')$, then $\P \leq_W \Q^{\rfn}$.
  In particular, if a theory $T \subseteq \{\sigma : \text{any $\omega$-model of $\L_2(\theta',\eta') + \ACAo$ satisfies $\sigma$} \}$ proves $\L_2(\theta,\eta)$,
   then $\P \leq_W \Q^{\rfn}$.
\end{theorem}
\begin{proof}
  Immediate from Lemma \ref{P-ref and omega-model reducibility} and \ref{P-ref-strengthen}.
\end{proof}

The previous theorem suggests that the $\omega$-model reflection of an arithmetically representable problem
is strong enough to capture provability from the problem.

Next, we see the separation of arithmetically representable problems and their reflections.

\begin{lemma}[$\omega$-model incompleteness, {\cite[Theorem VIII.5.6]{MR2517689}}]\label{omega-model incompleteness}

  Let $\sigma$ be an $\L_2$ sentence which is true in the intended model.
  Put $\mathbf{T} = \ACAo + \sigma$. Then, there is a countable $\omega$-model of
  \begin{align*}
    \mathbf{T} + \lnot \exists \text{countable coded $\omega$-model of $\mathbf{T}$}.
  \end{align*}
\end{lemma}

\begin{theorem}\label{Thm of omega-model incomp in Weih}
  If $\P$ has an arithmetical representation, then $\P^{\rfn} \not \leq^a_{W} \P$.
\end{theorem}
\begin{proof}
  Let $(\theta,\eta)$ be an arithmetical representation for $\P$. Then, $\P^{\rfn}$ has the following arithmetical representation $(\theta',\eta')$: $\theta'(X) \equiv X = X$ and
  \begin{align*}
    \eta'(X,(M,f,g,e)) \equiv\,
    &M_e = X,\\
    &\forall i \forall Y \leq_T M_i \exists j (Y = M_j), \\
    &f \text{ is a function from $\N$ to $\N$}, \\
    &\forall i (M_{f(i)} = (M_i)'),\\
    &g \text{ is a function from $\N^2$ to $\N$},\\
    &\forall i,j (M_{g(i,j)} = M_i \oplus M_j),\\
    &\forall i(\theta(M_i) \to \exists j \eta(M_j)).
  \end{align*}
  Let $\mathbf{T} = \ACAo + \L_2 (\theta,\eta)$. Since $\L_2(\theta,\eta)$ is a true sentence, we can find a countable $\omega$-model $M = \langle M_i \rangle_i$ such that
  \begin{align*}
    \langle M_i \rangle_i \models \mathbf{T} + \lnot \exists \text{countable coded $\omega$-model of $\mathbf{T}$}.
  \end{align*}

  First, we show that $M \not \models \L_2(\theta',\eta')$.
  Assume $M \models \L_2(\theta',\eta')$.
  Then, for any $X \in M$, there exist $A = \langle A_i \rangle_i, f,g \in M$ and $e \in \omega$ such that $M \models \eta'(X,(A,f,g,e))$. Thus,
  \begin{align*}
    M \models & A_e = X \ \land \\
    &\forall i \forall Y \leq_T A_i \exists j (Y = A_j) \ \land \\
    & \forall i (A_{f(i)} = A_i')  \ \land  \\
    &\forall i,j (A_{g(i,j)} = A_i \oplus A_j) \ \land \\
    & \forall i (\theta(A_i) \to \exists j \eta(A_i,A_j)).
  \end{align*}
  Pick a tuple $(A,f,g,e)$ satisfying $\eta'(\varnothing,(A,f,g,e))$. We claim that $M \models [A \models \mathbf{T}]$.
  Since $\mathbf{T}$ is finitely axiomatizable and $M$ is a model of $\ACAo$, it is enough to show that $M$ satisfies the relativization $\mathbf{T}^A$ of $\mathbf{T}$ to $A$.
  This is immediate from the definition of $A$.
  However, $M \not \models  \exists N [N \models \mathbf{T}]$ from the definition, so this is a contradiction.

  We next show that  $\P^{\rfn} \leq_W^a \P$ implies $M \models \L_2(\theta',\eta')$.
  Since $M \not \models \L_2(\theta',\eta')$, this yields that $\P^{\rfn} \not \leq_W^a \P$.
  Assume $\P^{\rfn} \leq_W^a \P$ via arithmetical operators $\Phi$ and $\Psi$.
  To see $M \models \L_2(\theta',\eta')$, pick an arbitrary $X \in M$.
  Then, we have
  \begin{align*}
    \theta(\Phi(X)) \land
    \forall Y \in \P(X) (\Psi(X,Y) \in \P^{\rfn}(X)).
  \end{align*}
  Since $\Phi$ is arithmetical, $\Phi(X) \in M$. Moreover, $M \models \theta(\Phi(X))$ holds because $\theta$ is arithmetical.
  Thus $M \models \exists Y \eta(X,Y)$ because $M \models \L_2(\theta,\eta)$.
  Take such a $Y$, then $Y \in \P(X)$ because $\eta$ is also arithmetical.
  Now we have $\Psi(X,Y) \in \P^{\rfn}(X) \cap M$ but this implies $M \models \eta'(X,\Psi(X,Y))$.
  Therefore, $M$ satisfies $\forall X \exists Z \eta'(X,Z)$.
\end{proof}

\begin{corollary}\label{rfn is enough strong}
  If $\P$ has an arithmetical representation, then $\P \leq_W \P^{\rfn}$ but $\P^{\rfn} \not \leq_W^a \P$.
\end{corollary}
\begin{proof}
  It is immediate from the above theorem and Remark \ref{four reductions}.
\end{proof}

Next, we see the relationship between $\omega$-model reflections and closed choice on Baire space $\CNN$.

\begin{definition}
  We define the problem $\Sigma^1_1\mathchar`-\mathsf{C}_{2^\omega}$ as follows:
  \begin{itembox}[l]{$\Sigma^1_1\mathchar`-\mathsf{C}_{2^\omega}$}
    \begin{description}
      \item[\sf Input] A $\Sigma^1_1$ formula $\varphi(X,Y)$ and a set $A$ such that $\exists X \varphi(X,A)$ holds.
      \item[\sf Output] A set $B$ such that $\varphi(B,A)$ holds.
    \end{description}
  \end{itembox}
\end{definition}
\begin{lemma}
  $\CNN$ is Weihrauch equivalent to $\Sigma^1_1\mathchar`-\mathsf{C}_{2^\omega}$.
\end{lemma}
\begin{proof}
  See Theorem IV.13 of \cite{PaulElliot}.
\end{proof}

\begin{theorem}
  The following hold.
  \begin{enumerate}
    \item If $\P$ is $(\Sigma^1_\infty,\Sigma^1_0)$-representable, then $\P \leq_W \CNN$.
    \item If $\P$ is $(\Pi^1_1,\Sigma^1_0)$-representable, then $\P <_W \CNN$.
  \end{enumerate}
\end{theorem}
\begin{proof}
  We show (1.). For (2.), see also \cite[Proposition 8.4]{arXiv1905.06868}.

  Let $(\theta,\eta)$ be a $(\Sigma^1_\infty,\Sigma^1_0)$-representation for $\P$. Let $X$ be an input for $\P$.
  Then, as we mentioned in Remark \ref{remark of input/output of problems}, $\{Y: \eta(X,Y)\} \neq \varnothing$.
  Thus, by $\Sigma^1_1\mathchar`-\mathsf{C}_{2^\omega}$, we can find a $Y$ such that $\eta(X,Y)$ holds.
\end{proof}

\begin{corollary}\label{reflection and CNN}
  If $\P$ has an arithmetical representation, then $\P^{\rfn} <_W \CNN$.
\end{corollary}
\begin{proof}
  If $\P$ has an arithmetical representation, then so does $\P^{\rfn}$.
\end{proof}

\subsection{$\omega$-model reflection of $\ATR,\ATR_2$}
In this subsection, we consider $\ATR$ and $\ATR_2$ by using $\omega$-model reflections.

\begin{lemma} [Folklore]
  Let $\mathrm{HYP}$ be the set of all hyperarithmetical sets. Then, there is no arithmetical representation ($\theta,\eta$) such that
  $\mathrm{HYP}$ is the smallest $\omega$-model of $\L_2(\theta,\eta)$.
\end{lemma}
\begin{proof}
  Let $(\theta,\eta)$ be an arithmetical representation such that $\mathrm{HYP} \models \L_2(\theta,\eta)$. We show that there is an $\omega$-model of $\L_2(\theta,\eta)$ smaller than $\mathrm{HYP}$.

  We note that $\L_2(\theta,\eta)$ is a $\Pi^1_2$ sentence and $\mathrm{HYP}$ is a model of $\Pi^1_2$-reflection (see Theorem VIII.4.11 and VIII.5.12 of \cite{MR2517689}. See also Definition \ref{def of reflection in SoA} in this paper).
  Thus, there exists $\langle M_i \rangle_i \in \HYP$ such that $\HYP \models [\langle M_i \rangle_i \models \L_2(\theta,\eta)]$.
  Since $\theta$ and $\eta$ are arithmetical, $\langle M_i \rangle_i$ is an $\omega$-model of $\L_2(\theta,\eta)$.
  Moreover, $\langle M_i \rangle_i$ is a proper subset of $\HYP$. This completes the proof.
\end{proof}

\begin{theorem}
  There is no arithmetically representable problem $\P$ such that $\P \equiv_W^a \ATR$.
\end{theorem}
\begin{proof}
  Assume $\P$ is an arithmetically representable problem such that $\P \equiv_W^a \ATR$.
  Then $\P^\rfn$ and $\ATR^\rfn$ are the same function by Theorem \ref{well-def of reflection}.
  We note that
  \begin{enumerate}
    \item There is an output $(\langle M_i \rangle_i,f,g,e) \in \ATR^\rfn(\varnothing)$ such that $\langle M_i \rangle_i = \HYP$, and
    \item for any $(\langle M_i \rangle_i,f,g,e) \in \ATR^\rfn(\varnothing)$, $\HYP \subseteq \langle M_i \rangle_i$.
  \end{enumerate}
  Since $\P^\rfn$ and $\ATR^\rfn$ are the same, these conditions also hold for $\P^\rfn$.
  Thus, we can pick a tuple $(\langle M_i \rangle_i,f,g,e) \in \P^\rfn(\varnothing)$ such that $\langle M_i \rangle_i = \HYP$.
  Then, by the previous theorem, we can find a tuple $(\langle N_i \rangle_i, \widetilde{f},\widetilde{g},\widetilde{e}) \in \P^\rfn(\varnothing)$ such that $\langle N_i \rangle_i \subsetneq \HYP$. However, this contradicts the second condition.
\end{proof}
This theorem also shows that there is no theory of hyperarithmetical analysis defined by a $\Pi^1_2$ formula.

We note that $\ATR_2$ is arithmetically representable. Thus, we can apply Theorem \ref{P-ref and provability} and Corollary \ref{reflection and CNN} to $\ATR_2$.
\begin{theorem}\label{P leq Prfn leq ATR_2rfn}
  The following holds.
  \begin{enumerate}
    \item Let $(\theta,\eta)$ be a $(\Pi^1_1,\Sigma^1_0)$-representation.
    If there is a theory $T$ such that
    \begin{align*}
      T \subseteq \{\sigma : \text{any $\omega$-model of $\ATRo$ satisfies $\sigma$} \} \text{ and }
    T \vdash \L_2(\theta,\eta),
      \end{align*}
      then $\Pb(\theta,\eta) \leq_W (\Pb(\theta,\eta))^{\rfn} \leq_W \ATR_2^{\rfn}$.
    \item  $\ATR_2^{\rfn} <_W \CNN$.
  \end{enumerate}
\end{theorem}
\begin{proof}
  (1.) is immediate from Theorem \ref{P-ref and provability} and
  (2.) is immediate from Theorem \ref{reflection and CNN}.
\end{proof}

\begin{remark}
  In (1.) of  Theorem \ref{P leq Prfn leq ATR_2rfn}, the assumption that $(\theta,\eta)$ is $(\Pi^1_1,\Sigma^1_0)$ is essential.
  In fact, open determinacy on Baire space is provable from $\ATRo$ \cite[Theorem V.8.7]{MR2517689}, Theorem V.8.7 but a variant of it is strictly stronger than $\CNN$ \cite[Corollary 7.5]{MR4231614}.
\end{remark}

\section{More on $\omega$-models}
In this section, we introduce a few variants of effective $\omega$-models.

We recall the transfinite induction in second-order arithmetic.
\begin{definition}
  Let $\varphi$ be a formula.
  We define transfinite induction on $\varphi$, $\varphi\myhyphen\TI$, as follows:
  \begin{equation*}
    \forall W \biggl(\WO(W) \to
    \Bigl(
    \forall i \in |W| \bigl((\forall j <_W i \varphi(j)) \to \varphi(i)\bigr)
    \to \forall i \in |W| \varphi(i)
    \Bigr)
    \biggr).
  \end{equation*}
  Let $\Pi^1_1\myhyphen\TI$ denote the set of transfinite inductions on any $\Pi^1_1$ formula.
\end{definition}

Next, we introduce the second-order arithmetic version of $\omega$-model reflections.
\begin{definition}\label{def of reflection in SoA}
  Let $\varphi$ be a formula.
  We define the $\omega$-model reflection on $\varphi$, $\varphi\myhyphen\mathbf{RFN}$ as follows:
  \begin{equation*}
    \forall X (\varphi(X) \to \exists \M:\text{coded $\omega$-model}(X \in \M \land \M \models \varphi(X))).
  \end{equation*}
  Let $\Pi^1_2\myhyphen\mathbf{RFN}$ denote the set of the $\omega$-model reflection on any $\Pi^1_2$ formula.
\end{definition}

\begin{lemma}[{\cite[VIII.5.12]{MR2517689}}]\label{lemma of equiv of TI and RFN}
  Over $\ACAo$, $\Pi^1_1\myhyphen\TI$ is equivalent to $\Pi^1_2\myhyphen\mathbf{RFN}$.
\end{lemma}

We consider effective $\omega$-models of $\Pi^1_1\myhyphen\TI$.
\begin{definition}
  Let $\P$ be a problem.
  We define the $\omega$-model reflection of $\P + \Pi^1_1\myhyphen\TI$ as follows:
  \begin{itembox}[l]{$(\P + \Pi^1_1\myhyphen\TI)^{\rfn}$}
  \begin{description}
      \item[\sf Input] Any set $X$.
      \item[\sf Output] A tuple $\mathcal{M} = (\langle M_i \rangle_i,f,g,e)$ such that
      \begin{align*}
          &(\langle M_i \rangle_i,f,g) \text{ is an effective $\omega$-model of $\ACAo$}, \\
          &M_e = X,\\
          &\forall i(M_i \in \dom \P \to \exists j (M_j \in \P(M_i))), \\
          &\M \models \Pi^1_1\myhyphen\TI.
      \end{align*}
  \end{description}
\end{itembox}
\end{definition}
Thus, $(\P + \Pi^1_1\myhyphen\TI)^\rfn$ is a stronger variant of $\P^\rfn$.
We will see that for any arithmetically representable problem $\P$, $(\P + \Pi^1_1\myhyphen\TI)^\rfn$ is much stronger than $\P^\rfn$.

\begin{lemma}\label{reflection of TI}
  Let $\P$ be an arithmetically representable problem.
  Then $(\P^{\rfn} + \Pi^1_1\myhyphen\TI)^\rfn \leq_W (\P + \Pi^1_1\myhyphen\TI)^\rfn$.
\end{lemma}
\begin{proof}
  Let $A$ be an input for $(\P^{\rfn} + \Pi^1_1\myhyphen\TI)^\rfn$ and
  $\M = (\langle M_i \rangle_i,f,g,e)$ be an output of $(\P + \Pi^1_1\myhyphen\TI)^\rfn(A)$.
  We show that $\M \in (\P^{\rfn} + \Pi^1_1\myhyphen\TI)^\rfn(A)$.
  Since $\M \models \Pi^1_1\myhyphen\TI$ by definition, it is enough to show that
  for any $X \in \langle M_i \rangle_i$, there exists $\mathcal{N} = (\langle N_i \rangle_i,\widetilde{f},\widetilde{g},\widetilde{e}) \in \langle M_i \rangle_i$ such that  $\mathcal{N} \in \P^{\rfn}(X)$.

  Let $(\theta,\eta)$ be an arithmetical representation for $\P$. Now $\M \models \L_2(\theta,\eta) + \Pi^1_2\mathchar`-\mathsf{RFN}$, and $\L_2(\theta,\eta)$ is a $\Pi^1_2$ sentence.
  Thus we have
  \begin{align*}
    \M \models \forall X \exists \langle N_i \rangle_i (X \in \langle N_i  \rangle_i \land \langle N_i \rangle_i \models \ACAo + \L_2(\theta,\eta)).
  \end{align*}
  We note that $X \in \langle N_i  \rangle_i \land \langle N_i \rangle_i \models \ACAo + \L_2(\theta,\eta)$ can be written as the following arithmetical formula
  $\xi(X,\langle N_i\rangle_i)$:
 \begin{align*}
   \xi(X,\langle N_i\rangle_i) \equiv \,
   &X \in \langle N_i \rangle_i,\\
   &\forall i \forall Y (Y \leq_T N_i \to Y \in \langle N_i \rangle_i), \\
   &\forall i \exists j (N_j = N_i'), \\
   &\forall i,j \exists k (N_k = N_i \oplus N_j), \\
   &\forall i (\theta(N_i) \to \exists j  \eta(N_i,N_j)).
 \end{align*}
  Take arbitrary $X$ and $\langle N_i\rangle_i \in \M$ such that  $\M \models \xi(X,\langle N_i\rangle_i)$.
  Then $\xi(X,\langle N_i\rangle_i))$ is also true in the intended model.
  Now
  \begin{align*}
    \M \models \exists \widetilde{e} (X = N_{\widetilde{e}}) \land \exists \widetilde{f} \forall i (N_{\widetilde{f}(i)} = N_i') \land \exists \widetilde{g} \forall i,j (N_{\widetilde{g}(i,j)} = N_i \oplus N_j).
  \end{align*}
  Pick such $\widetilde{f},\widetilde{g}$ and $\widetilde{e}$, then
 $\mathcal{N} = (\langle N_i \rangle_i,\widetilde{f},\widetilde{g},\widetilde{e}) \in \langle M_i \rangle_i \cap \P^{\rfn}(X)$.
\end{proof}

\begin{theorem}\label{Thm equiv of ref TI}
  Let $\P$ be an arithmetically representable problem.
  Then $(\P^{\rfn} + \Pi^1_1\myhyphen\TI)^\rfn \equiv_W (\P + \Pi^1_1\myhyphen\TI)^\rfn$.
\end{theorem}
\begin{proof}
  It is enough to show that $(\P^{\rfn} + \Pi^1_1\myhyphen\TI)^\rfn \geq_W (\P + \Pi^1_1\myhyphen\TI)^\rfn$.
  Take a set $X$ and $\M = (\langle M_i\rangle_i,f,g,e) \in (\P^{\rfn} + \Pi^1_1\myhyphen\TI)^\rfn(X)$.
  We show that $\M \in (\P + \Pi^1_1\myhyphen\TI)^{\rfn}(X)$.
  It is enough to show that
  \begin{align*}
    \forall i(M_i \in \dom(\P) \to \exists j (M_j \in \P(M_i)).
  \end{align*}
  Let $M_{i_0}$ be such that $M_{i_0} \in \dom(\P)$.
  Since $\M$ is an output of $(\P^{\rfn} + \Pi^1_1\myhyphen\TI)^{\rfn}$ and $M_{i_0}$ is an input for $\P^{\rfn}$,
  there exists $M_n = \langle M_{n,k} \rangle_k$ such that
  \begin{align*}
    &M_{i_0} \in \langle M_{n,k} \rangle_k, \\
    &\forall k (M_{n,k} \in \dom \P \to \exists l (M_{n,l} \in \P(M_{n,k}))).
  \end{align*}
  Since $M_{i_0} \in \dom \P$, there exists $l$ such that $M_{n,l} \in \P(M_{i_0})$. Take such an $l$.
  Then $M_{n,l} \leq_T M_n$ and hence $M_{n,l} \in \langle M_i \rangle_i$. This completes the proof.
\end{proof}

\begin{definition}
  For a problem $\P$, we define $\P^{(n\mathchar`-\rfn)}$ as follows.
  \begin{align*}
    &\P^{(0\mathchar`-\rfn)} = \P,\\
    &\P^{(n+1\mathchar`-\rfn)} = (\P^{(n\mathchar`-\rfn)})^{\rfn}.
  \end{align*}
\end{definition}
We remark that if $\P$ is an arithmetically representable problem, then so is $\P^{\rfn}$.
Thus for any arithmetically representable problem $\P$ and any $n \in \omega$,  $\P^{(n \mathchar`- \rfn)} \leq_W \P^{(n+1 \mathchar`-\rfn)}$ but $\P^{(n+1\mathchar`-\rfn)} \not \leq_{W}^a \P^{(n\mathchar`-\rfn)}$
(see Corollary \ref{rfn is enough strong}).

\begin{corollary}\label{rfn iteration TI}
  Let $\P$ be an arithmetically representable problem.
  Then, for any $n$, $\P^{(n \mathchar`- \rfn)} <_W (\P^{(n \mathchar`- \rfn)} + \Pi^1_1\myhyphen\TI)^{\rfn} \equiv_W (\P + \Pi^1_1\myhyphen\TI)^{\rfn}$.
\end{corollary}
\begin{proof}
  Let $\P$ be an arithmetically representable problem.
  Recall that for any arithmetically representable problem, its reflection is also arithmetically representable.
  Thus, $\P^{(n\mathchar`-\rfn)}$ is also arithmetically representable for any $n$.

  The inequality $\P^{(n \mathchar`- \rfn)} <_W \P^{(n+1 \mathchar`- \rfn)}$ follows from Corollary \ref{rfn is enough strong}. Thus, $\P^{(n \mathchar`- \rfn)} <_W \P^{(n+1 \mathchar`- \rfn)} \leq_W (\P^{(n \mathchar`- \rfn)} + \Pi^1_1\myhyphen\TI)^{\rfn}$.

  For the equivalence $(\P^{(n \mathchar`- \rfn)} + \Pi^1_1\myhyphen\TI)^{\rfn} \equiv_W (\P + \Pi^1_1\myhyphen\TI)^{\rfn}$, we use Theorem \ref{Thm equiv of ref TI} repeatedly.
  Since each $\P^{(n\mathchar`-\rfn)}$ is arithmetically representable,
  we have $(\P^{(n+1 \mathchar`- \rfn)} + \Pi^1_1\myhyphen\TI)^{\rfn} \equiv_W (\P^{(n \mathchar`-\rfn)} + \Pi^1_1\myhyphen\TI)^{\rfn}$ for each $n$.
\end{proof}

Next, we consider weaker versions of $\beta$-models.
This will be used to show the strictness of the hierarchy of $\LPP$s in the next section.
For the basic things for $\beta$-models, see \cite{MR2517689}.

\begin{definition}
  Let $X$ be a set.
  We write $Y \in \Delta^0_n(X)$ if $Y$ is $\Delta^0_n$ definable from $X$.
  Therefore, $Y \in \Delta^0_{n+1}(X)$ if and only if $Y$ is computable from $X^{(n)}$, the $n$-th Turing jump of $X$.
  We also write $Y \in \Sigma^0_n(X)$ if $Y$ is $\Sigma^0_n$ definable from $X$.
\end{definition}

\begin{definition}
  Let $\M = (\langle M_i \rangle_i,f,g,e)$ be an effective $\omega$-model.
  We say $\M$ is a $\Delta^0_n \beta$-model if
  for any $\Pi^0_2$ formula $\varphi(X,\vec{V})$ and $\vec{A} \in \M$,
  \begin{equation*}
     \exists X \in \Delta^0_n(\M) \varphi(X,\vec{A})
     \Leftrightarrow
     \M \models \exists X \varphi(X,\vec{A}).
  \end{equation*}
\end{definition}

\begin{remark}
  Since $\varphi$ is arithmetical in the above definition, the condition
  \begin{equation*}
     \exists X \in \Delta^0_n(\M) \varphi(X,\vec{A})
     \Leftrightarrow
     \M \models \exists X \varphi(X,\vec{A})
  \end{equation*}
  is equivalent to
  \begin{equation*}
     \exists X \in \Delta^0_n(\M) \varphi(X,\vec{A})
     \Rightarrow
     \M \models \exists X \varphi(X,\vec{A}).
  \end{equation*}
\end{remark}

\begin{lemma}
  Let $\M$ be an effective $\omega$-model.
  Then, $\M$ is a $\Delta^0_n \beta$-model if and only if for any linear order $L \in \M$
  \begin{equation*}
     \exists f \in \Delta^0_n(\M) (f \text{ is a descending sequence of $L$})
     \Leftrightarrow
     \M \models \lnot \WO(L).
  \end{equation*}
\end{lemma}
\begin{proof}
  Let $\M$ be an effective $\omega$-model.
  We note that the condition `$f$ is a descending sequence of $L$' is $\Pi^0_2$. Thus,  if $\M$ is a $\Delta^0_n\beta$-model and $L \in \M$ is a linear order, then
  \begin{equation*}
     \exists f \in \Delta^0_n(\M) (f \text{ is a descending sequence of $L$})
     \Leftrightarrow
     \M \models \lnot \WO(L).
  \end{equation*}

  Conversely, assume that
  for any linear order $L \in \M$,
  \begin{equation*}
     \exists f \in \Delta^0_n(\M) (f \text{ is a descending sequence of $L$})
    \Leftrightarrow
    \M \models \lnot \WO(L)
  \end{equation*}
  and show that $\M$ is a $\Delta^0_n \beta$-model.
  Let $\varphi(X,\vec{V}) \equiv \forall x \exists y \theta(x,y,X,\vec{V})$ be a $\Pi^0_2$ formula.
  We will show that for any $\vec{A} \in \M$, if $\exists X \in \Delta^0_n(\M) \varphi(X,\vec{A})$, then $\M \models \exists X \varphi(X,\vec{A})$.

  Let $\vec{A} \in \M$.
  Let $T \subseteq (2 \times \N)^{<\N}$ be an $\vec{A}$-computable tree such that $(X,f) \in [T] \Leftrightarrow \forall x \theta(x,f(x),X,\vec{A})$ for any $X$ and $f$.
  Then, any $\omega$-model of $\ACAo$ including $\vec{A}$ satisfies that
  \begin{align*}
    \lnot \WO(\KB(T)) \Leftrightarrow \exists X \varphi(X,\vec{A})
    \end{align*}
  where $\KB$ denotes the Kleene-Brouwer ordering.
  Assume $\exists X \in \Delta^0_n(\M) \varphi(X,\vec{A})$.
  Let $X \in \Delta^0_n(\M)$ such that $\varphi(X,\vec{A})$ holds. Then, $\forall x \exists y \theta(x,y,X,\vec{A})$ holds.
  Define $f \leq_T X \oplus \vec{A}$ by $f(x) = \mu y. \theta(x,y,X,\vec{A})$.
  Then $(X,f)$ is a path of $T$ and hence a descending sequence of $\KB(T)$.
  Since $X$ is $\Delta^0_n(\M)$ and $\vec{A}$ is $\Delta^0_0(\M)$, $(X,f)$ is $\Delta^0_n(\M)$.
  Thus $(X,f)$ is a $\Delta^0_n(\M)$-definable descending sequence of $\KB(T)$.
  By assumption $\M \models \lnot \WO(\KB(T))$ and hence $\M \models \exists X \varphi(X)$.
\end{proof}

\begin{lemma}\label{Delta0nbeta model is a model of Pi11ti}
  Let $\M$ be a $\Delta^0_1 \beta$-model. Then $\M \models \Pi^1_1\myhyphen\TI$.
\end{lemma}
\begin{proof}
  Let $W \in \M$ such that $\M \models \WO(W)$.
  Let $\varphi(i) \equiv \forall X \theta(i,X)$ be a $\Pi^1_1$ formula.
  We will show that $\M$ satisfies transfinite induction along $W$ on $\varphi$.
  For the sake of contradiction assume
  \begin{equation*}
    \M \models \forall i \in |W| [(\forall j<_W i \varphi(j)) \to \varphi(i)]
  \end{equation*}
  but
  \begin{equation*}
    \M \models \lnot (\forall i \in |W| \varphi(i)).
  \end{equation*}

  We may assume that $\varphi$ has exactly one set parameter $A = M_a$.
  Then, there exist a $\Sigma^0_0$ formula $\theta_0$ and $k \in \omega$ such that
  \begin{equation*}
    \{i \in |W| : \M \models \lnot \varphi(i)\} =
    \{i \in |W| : \exists b \theta_0(i,M_{f^kg(a,b)})\}
  \end{equation*}
  and hence $\{i \in |W| : \M \models \lnot \varphi(i)\}$ is $\Sigma^0_1(\M)$.
  Since $\{i \in |W| : \M \models \lnot \varphi(i)\}$ is nonempty and has no $W$-minimal element, we can find, in $\Delta^0_1(\M)$, a $W$-descending sequence from $\{i \in |W| : \M \models \lnot \varphi(i)\}$.
  Since $\M$ is a $\Delta^0_1\beta$-model, $\M \models \lnot \WO(W)$ but this is a contradiction.
\end{proof}

\section{Relativized leftmost path principle}
In the previous two sections, we introduced the $\omega$-model reflection and its variations.
It was shown that $\ATR_2^{\rfn}$ is strong enough to capture the provability from $\ATRo$ but still weaker than $\CNN$.
In this section, we seek for more concrete problems between $\ATR_2^{\rfn}$ and $\CNN$, and see
the hierarchical structure of them by using $\omega$-model reflections.

In his paper \cite{MR3145191}, Towsner introduced a hierarchy above $\ATRo$ in the context of reverse mathematics.
We rephrase his work in the context of Weihrauch degrees.

\begin{definition}
  Let $T$ be an ill-founded tree and $f,g \in [T]$. We say $f$ is lexicographically smaller than $g$, written $f <_l g$, if
  $(\exists n)(f[n] = g[n] \land f(n) < g(n))$.
\end{definition}

\begin{definition}
  Let $k \in \omega$. We define $\Delta^0_k\LPP$, the $\Delta^0_k$-relativized leftmost path principle, as follows.
  \begin{itembox}[l]{$\Delta^0_k\LPP$}
  \begin{description}
    \item[\sf Input] An ill-founded tree $T \subseteq \omega^{<\omega}$ and its path $f$.
    \item[\sf Output] A path $g$ of $T$ such that there is no $\Delta^0_k(T\oplus f \oplus g)$ path $h$ of $T$ with $h <_l g$.
  \end{description}
\end{itembox}
\end{definition}

\begin{remark}
  We note that a function is $\Delta^0_k$ definable if and only if it is $\Sigma^0_k$ definable. Thus, our $\Delta^0_k\LPP$ is the same as Towsner's $\Sigma^0_k\LPP$.
\end{remark}

\begin{remark}
  If we define $\Delta^0_k\LPP$ as follows, then it becomes equivalent to $\CNN$, so we do not adopt this definition.
  \begin{screen}
  \begin{description}
    \item[\sf Input] An ill-founded tree $T \subseteq \omega^{<\omega}$.
    \item[\sf Output] A path $g$ of $T$ such that there is no $\Delta^0_k(T\oplus g)$ path $h$ of $T$ with $h <_l g$.
  \end{description}
\end{screen}
\end{remark}

\begin{lemma}\label{A new version of LPP}
  Define $\Delta^0_k\overline{\LPP}$ as follows. Then, $\Delta^0_k\LPP$ is Weihrauch equivalent to $\Delta^0_k\overline{\LPP}$.
  \begin{itembox}[l]{$\Delta^0_k\overline{\LPP}$}
  \begin{description}
    \item[\sf Input] An ill-founded tree $T \subseteq \omega^{<\omega}$ and its path $f$.
    \item[\sf Output] A path $g$ of $T$ such that there is no $\Delta^0_k(g)$ path $h$ of $T$ with $h <_l g$.
  \end{description}
\end{itembox}
\end{lemma}
\begin{proof}
  It is trivial that $\Delta^0_k\overline{\LPP} \leq_W \Delta^0_k\LPP$.
  Thus, we will show that the converse direction holds.

  First, we introduce some notations.
  For $\sigma \in \omega^{<\omega}$,
  define $e_\sigma$ as the maximum even number smaller than $|\sigma|$
  and $o_\sigma$ as the maximum odd number smaller than $|\sigma|$.
  We also define
  $\sigma_{\even} = \langle \sigma(0),\ldots,\sigma(e_\sigma) \rangle,
  \sigma_{\odd} = \langle \sigma(1),\ldots,\sigma(o_\sigma) \rangle$.
  Moreover, for $f \in \omega^\omega$, define
  $f_{\even} = \langle f(2n) \rangle_n, f_{\odd} = \langle f(2n+1) \rangle_n$.
  Finally, for $f,g \in \omega^\omega$, define $f \sqcup g$ as the function such that
  $(f \sqcup g)_\even = f$ and $(f \sqcup g)_\odd = g$.

  Let $(T,f)$ be an input for $\Delta^0_k\LPP$.
  Define a tree $S$ by
  \begin{equation*}
    \sigma \in S \leftrightarrow \sigma_\even \in T \land \sigma_\odd \prec (T \oplus f).
  \end{equation*}
  Then, we have
  \begin{itemize}
    \item $\forall h \in [T](h \sqcup (T \oplus f) \in [S])$. In particular, $f \sqcup (T \oplus f)$ is a path of $S$.
    \item $\forall h \in [S](h_\even \in [T] \land h_\odd = T \oplus f)$.
          In particular, for any $g,h \in [S]$, $g \leq_l h$ if and only if $g_\even \leq_l h_\even$.
  \end{itemize}

  Now we can apply $\Delta^0_k\overline{\LPP}$ to $S$ and $f \sqcup (T \oplus f)$.
  Let $g \in [S]$ such that  $\forall h \in \Delta^0_k(g) \cap [S] (g \leq_l h)$.
  \setcounter{claimcounter}{1} 
  \begin{claim}
  $g_\even$ is an output of $\Delta^0_k\LPP(T,f)$.
  \end{claim}
  \addtocounter{claimcounter}{1}
  \begin{proof of claim}
    Let $h \in \Delta^0_k(T \oplus f \oplus g_\even) \cap [T]$.
    Then, since $T \oplus f \oplus g_\even$ is computable from $g$,
    $h \sqcup (T \oplus f) \in \Delta^0_k(g) \cap [S]$.
    Thus, $g \leq_l h \sqcup (T \oplus f)$. Hence $g_\even \leq_l (h \sqcup (T \oplus f))_\even = h$.
    This completes the proof.
    \qedclaim
  \end{proof of claim}
  Therefore, we can find an output of $\Delta^0_k\LPP$ by using $\Delta^0_k\overline{\LPP}$.
\end{proof}
By the above lemma, we sometimes identify $\Delta^0_k\LPP$ and $\Delta^0_k\overline{\LPP}$.
We note that for any $k > 0$, $\Delta^0_k\LPP$ has the following arithmetical representation $(\theta,\eta)$:
\begin{align*}
  &\theta(T,f) \equiv T \text{ is a tree } \land f \in [T], \\
  &\eta(T,f,g) \equiv g \in [T] \land \forall h \leq_T (T \oplus f \oplus g)^{(k-1)} (h \in [T] \to g \leq_l h).
\end{align*}
Similarly, $\Delta^0_k\overline{\LPP}$ has an arithmetical representation.
In the following argument, we fix these representations.
As noted in Lemma \ref{A new version of LPP}, we may assume that
$\L_2(\Delta^0_k\LPP)$ is equivalent to $\L_2(\Delta^0_k\overline{\LPP})$ over $\ACAo$ under these representations.

We will show that $\LPP$s form a strictly increasing hierarchy between $\ATR_2^{\rfn}$ and $\CNN$.

\begin{lemma}\label{Delta0nbeta model is a model of LPP}
  Let $n > 0$ and $\M$ be a $\Delta^0_{n+2} \beta$-model.
  Then $\M \models \L_2(\Delta^0_n\LPP)$.
\end{lemma}
\begin{proof}
  Let $\M$ be a $\Delta^0_{n+2} \beta$-model.
  It is enough to show that for any tree $T \in \M$, if
  $\M \models [T] \neq \varnothing$, then
  \begin{equation*}
    \M \models \exists g \in [T] \forall f \in [T] \cap \Delta^0_n(g) (g \leq_l f).
  \end{equation*}

  Define $S = \{\sigma \in T : \M \models \exists f \in [T]( \sigma \prec f)\}$.
  Then $S \in \Sigma^0_1(\M)$ and $S =\{\sigma \in T : \exists f \in [T] \cap \Delta^0_{n+2}(\M) (\sigma \prec f)\}$.
  Since $S$ is a nonempty pruned tree, $S$ has a leftmost path $g$ which is computable from $S$. In particular, $g \in \Delta^0_2(\M)$.


  We point out that if $f \in [T] \cap \Delta^0_n(g)$, then $g \leq_l f$.
  Assume $f \in [T] \cap \Delta^0_n(g)$. Since $g \in \Delta^0_2(\M)$,
  $f \in \Delta^0_{n+1}(\M)$.
  Thus $f \in [S]$ by definition and thus $g \leq_l f$.
  Since $f \in \Delta^0_n(g)$ is equivalent to $f$ is computable from $g^{(n-1)}$,
  we have
  \begin{align*}
    \exists (g_0 \oplus g_1 \oplus \cdots \oplus g_{n}) \in \Delta^0_{n+2}(\M) [
    &\bigwedge_{i < n} g_{i+1} = (g_i)'\ \land \\
    &g_0 \text{ is a function} \land g_0 \in [T] \ \land \\
    &(\forall X \leq_T g_{n-1}) ((X \text{ is a function} \land X \in [T]) \to g_0 \leq_l X)].
  \end{align*}
  Here, $(g_i)'$ denotes the Turing jump of $g_i$.
  We see that the condition in the square bracket is $\Pi^0_2$ in $(g_0 \oplus g_1 \oplus \cdots \oplus g_{n} )$.
  Since $g_{i+1} = (g_i)'$ is a $\Delta^0_2$ condition, `$g_0$ is a function' is $\Pi^0_2$ and $g_0 \in [T]$ is $\Pi^0_1$,
  it is enough to see that the last clause is $\Pi^0_2$.
  Let $U(e,x,X)$ be a $\Sigma^0_1$ universal formula.
  Put $\theta(X) \equiv (X \text{ is a function} \land X \in [T]) \to g_0 \leq_l X$.
  Since $\theta(X)$ is a $\Sigma^0_2$ formula, there is a $\Sigma^0_2$ formula $\theta'(e,e',V)$ such that
  \begin{align*}
    \forall e,e',Y \Bigl( &\bigl( \forall x (U(e,x,Y) \leftrightarrow \lnot U(e',x,Y)) \bigr) \\
    \to &\bigl( \forall X(X = \{x : U(e,x,Y)\} \to (\theta(X) \leftrightarrow \theta'(e,e',Y)) \bigr) \Bigr).
  \end{align*}
  Then, the last clause can be written as
  \begin{align*}
    \forall e,e' \Bigl(
      \bigl(
        \forall x (U(e,x,g_{n-1}) \leftrightarrow \lnot U(e',x,g_{n-1}))
      \bigr)
    \to \theta'(e,e',g_{n-1})
    \Bigr).
  \end{align*}
  Since this condition is $\Pi^0_3$ in $g_{n-1}$ and $g_{n} = (g_{n-1})'$, this is $\Pi^0_2$ in $g_n$.
  Since $\M$ is a $\Delta^0_{n+2}\beta$-model, $\M$ satisfies
  \begin{align*}
      \exists (g_0 \oplus g_1 \oplus \cdots \oplus g_{n})  [
      &\bigwedge_{i < n} g_{i+1} = (g_i)'\ \land \\
      &g_0 \text{ is a function} \land g_0 \in [T] \ \land \\
      &(\forall X \leq_T g_{n-1}) ((X \text{ is a function} \land X \in [T]) \to g_0 \leq_l X)].
  \end{align*}
  Therefore,
  \begin{equation*}
    \M \models \exists g  [g \in [T] \land (\forall f \in [T] \cap \Delta^0_n(g)) (g \leq_l f)].
  \end{equation*}
  This completes the proof.
\end{proof}

\begin{definition}
  Let $n \in \omega$. We define the $\Delta^0_n\beta$-model reflection as follows:
  \begin{itembox}[l]{$\Delta^0_n\beta$-model reflection}
  \begin{description}
    \item [\sf Input] Any set $X$.
    \item [\sf Output] A tuple $\M = (\langle M_i \rangle_i,f,g,e)$ such that
    $(\langle M_i \rangle_i,f,g) \text{ is a  $\Delta^0_n\beta$-model}$ and $M_e = X$.
  \end{description}
\end{itembox}
\end{definition}

\begin{lemma}[{\cite[Theorem 4.3]{MR3145191}}]\label{LPP to TI}
  Let $n \in \omega$.
  Then the $\Delta^0_n\beta$-model reflection is Weihrauch reducible to $\Delta^0_n\LPP$.
\end{lemma}
\begin{proof}
  We can use the same construction in the proof of Theorem 4.3. of \cite{MR3145191}.
\end{proof}

\begin{lemma}[{\cite[Theorem 4.2]{MR3145191}}]\label{lpp and atr}
  Over $\RCAo$, $\L_2(\Sigma^0_0\LPP)$ implies $\ATRo$.
\end{lemma}

\begin{theorem}
  $(\ATR_2 + \Pi^1_1\myhyphen\TI)^{\rfn}$ is Weihrauch reducible to $\Delta^0_3\LPP$.
\end{theorem}
\begin{proof}
  By Lemma \ref{Delta0nbeta model is a model of Pi11ti}, \ref{Delta0nbeta model is a model of LPP} and \ref{LPP to TI}, we have
  $(\Delta^0_1\LPP + \Pi^1_1\myhyphen\TI)^{\rfn} \leq_W \Delta^0_3\LPP$.
  Since $\L_2(\Delta^0_1\LPP)$ implies $\L_2(\ATR_2)$ over $\ACAo$, we also have $(\ATR_2 + \Pi^1_1\myhyphen\TI)^{\rfn} \leq_W \Delta^0_3\LPP$ (see Theorem \ref{P-ref and provability}).
\end{proof}

\begin{corollary}
  For any $n \in \omega$,
  \begin{enumerate}
    \item $\ATR_2^{(n \mathchar`- \rfn)} <_W \Delta^0_3\LPP$,
    \item $\Delta^0_3\LPP \not \leq_{W}^a \ATR_2^{(n \mathchar`- \rfn)}$.
  \end{enumerate}
\end{corollary}
\begin{proof}
  By Lemma \ref{rfn iteration TI} and the previous proof, we have
  \begin{equation*}
    \ATR_2^{(n \mathchar`- \rfn)} <_W (\ATR_2 + \Pi^1_1\myhyphen\TI)^{\rfn} \leq_W (\Delta^0_1\LPP + \Pi^1_1\myhyphen\TI)^{\rfn} \leq_W \Delta^0_3 \LPP.
  \end{equation*}
  Since $\ATR_2^{(n +1 \mathchar`- \rfn)} \not \leq_{W}^a \ATR_2^{(n \mathchar`- \rfn)}$, we have $\Delta^0_3\LPP \not \leq_{W}^a \ATR_2^{(n \mathchar`- \rfn)}$.
\end{proof}

By Lemma \ref{Delta0nbeta model is a model of Pi11ti}, \ref{Delta0nbeta model is a model of LPP} and \ref{LPP to TI}, we also have the following theorem.
\begin{theorem}
  $(\Delta^0_n\LPP + \Pi^1_1\myhyphen\TI)^{\rfn}$ is Weihrauch reducible to $\Delta^0_{n+2}\LPP$.
  Therefore, for any $k \in \omega$,
  \begin{enumerate}
    \item $(\Delta^0_n\LPP)^{(k \mathchar`- \rfn)} <_W \Delta^0_{n+2}\LPP$,
    \item $\Delta^0_{n+2}\LPP \not \leq_{W}^a (\Delta^0_n\LPP)^{(k \mathchar`- \rfn)}$.
  \end{enumerate}
\end{theorem}

\begin{remark}
  Any arguments in this section can be verified in $\ACAo$.
  Thus, we also have a precise analysis for the hierarchy of $\Sigma^0_k\LPP$s in second-order arithmetic.
\end{remark}

\section{Relativized least fixed point}
We have studied two types of Weihrauch problems, $\FP$ and $\LPP$.
In this section, we see that $\LPP$ can be regarded as a variant of $\FP$.

First, we define an extension of $\FP$ and its variants.

\begin{definition}
    Let $i,j \in \omega$. We define $\Sigma^0_i\LFP(\Sigma^0_j)$,
    the $\Sigma^0_i$-relativized least fixed point theorem for $\Sigma^0_j$ operators, as follows.
    \begin{itembox}[l]{$\Sigma^0_i\LFP(\Sigma^0_j)$}
    \begin{description}
      \item[\sf Input] A set $A$ and a Gödel number $\GN{\varphi}$ of an $X$-positive formula $\varphi(n,X,Y) \in \Sigma^0_j$. (The same as $\FP(\Sigma^0_j)$).
      \item[\sf Output] A fixed point $X$ of the operator such that there is no $\Sigma^0_i(A \oplus X)$ definable fixed point $Y$ with $Y \subsetneq X$.
    \end{description}
  \end{itembox}
\end{definition}

\begin{definition}
  Let $\Gamma$ be an operator on $\mathcal{P}(\omega)$. We say a set $X$ is a closed point of $\Gamma$ if $\Gamma(X) \subseteq X$.
\end{definition}
\begin{definition}
    Let $i,j \in \omega$. We define $\Sigma^0_i\LCP(\Sigma^0_j)$,
    the $\Sigma^0_i$-relativized least closed point theorem for $\Sigma^0_j$ operators, as follows.
    \begin{itembox}[l]{$\Sigma^0_i\LCP(\Sigma^0_j)$}
    \begin{description}
      \item[\sf Input] A set $A$ and a Gödel number $\GN{\varphi}$ of an $X$-positive formula $\varphi(n,X,Y) \in \Sigma^0_j$. (The same as $\FP(\Sigma^0_j)$).
      \item[\sf Output] A closed point $X$ of the operator such that there is no $\Sigma^0_i(A \oplus X)$ definable closed point $Y$ with $Y \subsetneq X$.
    \end{description}
  \end{itembox}
\end{definition}

\begin{remark}
  We can define $\Omega\LFP(\Omega')$ and $\Omega\LCP(\Omega')$ for $\Omega,\Omega'\in \{\Sigma^0_i,\Delta^0_i,\Pi^0_i : i \in \omega\}$ in the same manner.
\end{remark}

\begin{remark}
  We note that the following assertions are mutually equivalent over $\ACAo$.
  \begin{enumerate}
    \item $\Pi^1_1\myhyphen\mathbf{CA}_0$.
    \item The existence of a least fixed point for operators defined by positive arithmetical formulas.
    \item The existence of a least closed point for operators defined by positive arithmetical formulas.
    \item The existence of a leftmost path for ill-founded trees.
  \end{enumerate}
  The equivalence of (1.) and (4.) is in \cite[Theorem 6.5]{marcone1996logical}.
  The equivalence of (2.) and (3.) is due to \cite{jagerfixedpoint}.
\end{remark}
We will show that the equivalences among (2.), (3.) and (4.) also hold for relativized versions in the sense of Weihrauch reducibility.

\begin{theorem}\label{from LFP to LPP}
  Let $k \in \omega$. Then, $\Delta^0_k\LPP \leq_W \Delta^0_k\LFP(\Pi^0_1)$.
\end{theorem}
\begin{proof}
  Let $T \subseteq \omega^{<\omega}$ be an ill-founded tree and $f$ be a path of $T$. We will find a $\Delta^0_k$-relativized leftmost path of $T$.
  In the following argument, we fix a computable bijection between $\omega$ and $\omega^{<\omega}$. Thus we can apply $\Delta^0_k\LFP(\Pi^0_1)$ to monotone operators on $\mathcal{P}(\omega^{<\omega})$.

  For each $X \subseteq \omega^{<\omega}$, define $S_X := \{\sigma: \exists \tau \in X (\tau \preceq \sigma) \}$.
  Define an operator $\Gamma$ by
  \begin{align*}
    \Gamma(X) = S_X \cup \{\sigma \in \omega^{<\omega} : (\exists \tau \preceq \sigma)(\tau \text{ is an endnode of } T - S_X) \}.
  \end{align*}
  We show that $\Gamma(X)$ is $\Pi^0_1$ and $X$-positive. Let $\sigma \ast \tau$ denote the concatenation of $\sigma$ and $\tau$.
  Define $\theta_1(\sigma),\theta_2(\sigma)$ and $\theta(\sigma)$ as follows:
  \begin{align*}
    \theta_1(\sigma) &\equiv \exists \tau \preceq \sigma (\tau \in X), \\
    \theta_2(\sigma) &\equiv \exists \tau \preceq \sigma (\tau \in T - S_X \land \forall n( \tau \ast \langle n \rangle \not \in T - S_X)), \\
    \theta(\sigma) &\equiv \theta_1(\sigma) \lor \theta_2(\sigma).
  \end{align*}
  Then $\theta$ is $\Pi^0_1$ in $T \oplus X$ because $S_X$ is $\Delta^0_0$ definable in $X$. In addition, $\theta$ is $X$-positive. It is easy to show that $\theta$ defines $\Gamma$.

\setcounter{claimcounter}{1}

\begin{claim}
  Let $X$ be a fixed point of $\Gamma$. Then
  $T-X$ has no endnode.
\end{claim}
\addtocounter{claimcounter}{1}
\begin{proof of claim}
  Let $X$ be a fixed point. Then $X \subseteq S_X \subseteq \Gamma(X) = X$.
  Thus $X = S_X = \Gamma(X)$.
  Let $\sigma \in T - X$. Then $\sigma \in T - S_X$.
  If there exists a $\tau \preceq \sigma$ such that $\tau$ is an endnode of $T - X = T - S_X$, then
  $\sigma \in \Gamma(X) = X$.
  However, this is a contradiction because $\sigma \in T - X$.
  Thus, any $\tau \preceq \sigma$ is not an endnode of $T - X$. In particular, $\sigma$ is not an endnode.
  \qedclaim
\end{proof of claim}

We note that the set parameter occurring in the definition of $\Gamma$ is just $T$. However, we may assume that both of $T$ and $f$ occur in the definition.
Therefore, we can apply $\Delta^0_k\LFP(\Pi^0_1)$ to $(\theta,T \oplus f)$.
\begin{claim}
  If $X$ is an output of $\Delta^0_k\LFP(\Pi^0_1)(\theta,T \oplus f)$, then $f \in [T-X]$.
\end{claim}
\addtocounter{claimcounter}{1}

\begin{proof of claim}
  At first we  show that if $X$ is a fixed point of $\Gamma$, then so is $X - \{f[n]: n \in \omega\}$.

  Let $X$ be a fixed point of $\Gamma$ and $Y = X - \{f[n]: n \in \omega\}$.
  Since $Y \subseteq \Gamma(Y)$ by definition of $\Gamma$, it is enough to show that $Y \supseteq \Gamma(Y)$.
  We first show that $Y = S_Y$. $Y \subseteq S_Y$ is by definition.
  Let $\sigma \in S_Y$. Then $\sigma \in S_X = X$. On the other hand, there exists a $\tau \preceq \sigma$ such that $\tau \in Y$. Take such a $\tau$.
  By definition of $Y$, $\tau \neq f[|\tau|]$. Since $\tau \preceq \sigma$, $\sigma \not \prec f$. So if $\sigma \in S_Y$, then $\sigma \in X$ but $\sigma \not \prec f$. This means $\sigma \in Y$. Therefore, $S_Y \subseteq Y$.

  Let $\sigma \in \Gamma(Y)$. Then, by definition of $\Gamma$,
\begin{itemize}
  \item $\sigma \in S_Y = Y $, or
  \item there is an initial segment $\tau$ of $\sigma$ which is an endnode of $T - S_Y = T- Y$.
\end{itemize}
  We show  that the second case never happens.
  Pick $\sigma \in \Gamma(Y)$ and $\tau \preceq \sigma$, and assume $\tau$ is an endnode of $T - Y$.
  We show that $\tau$ must be an endnode of $T - X$.
  Since $\tau \in T- Y$, $(\tau \in T) \land (\tau \prec f \lor \tau \not \in X)$.
  However, if $\tau \prec f$ then $\tau \prec f[|\tau|+1] \in T - Y$ but this is impossible because $\tau$ is an endnode of $T-Y$. So $\tau \in T - X$.
  Moreover, if $\tau \ast \langle n \rangle \in T-X$ for some $n$, then $\tau \ast \langle n\rangle \in T- Y$ but this never happens. So $\tau$ is an endnode of $T - X$.
  However, as we have already proved, $T-X$ has no endnode.

  Now $Y \subseteq X$ is a fixed point of $\Gamma$ such that $Y \in \Sigma^0_0(T \oplus f \oplus X)$ and $f \in [T-Y]$.
  If $X$ is an output of $\Delta^0_k\LFP(\Pi^0_1)(\theta,T \oplus f)$, then $X = Y$.
  Thus $f \in [T -X]$.
  \qedclaim
\end{proof of claim}
  Let $X$ be an output of $\Delta^0_k\LFP(\Pi^0_1)(\theta,T \oplus f)$.
  Then by the above claims, $T - X$ is a nonempty pruned tree.
  So there exists $g\leq_T T  \oplus X $ such that $g$ is the leftmost path of $T -X$.
  Assume this $g$ is not a $\Delta^0_k$ leftmost path of $(T,f)$.
  Then, we can pick $h \in [T] \cap \Delta^0_k(T,f,g)$ such that $h <_l g$.
  On the other hand, by the same way of the proof of the second claim, $X - \{h[n]: n \in \mathbb{N}\}$ is a fixed point smaller than $X$ definable in $\Delta^0_k(X)$.
  So this is a contradiction.
\end{proof}
\begin{remark}\label{rmk LPP and LCP}
  The operator $\Gamma$ in the above proof satisfies $\forall X(X \subseteq \Gamma(X))$.
  Thus, for any $X$, $X$ is a fixed point of $\Gamma$ if and only if $X$ is a closed point of $\Gamma$.
  Therefore, we also have $\Delta^0_k\LPP \leq_W \Delta^0_k\LCP(\Pi^0_1)$.
\end{remark}

\begin{remark}
  We note that the least fixed point for a $\Sigma^0_1$ monotone operator is $\Sigma^0_1$ definable from the operator.
  In fact, if $\Gamma$ is a $\Sigma^0_1$ monotone operator, then $\bigcup_n \Gamma^n(\varnothing)$ is the least fixed point.
  Moreover, clearly $\Sigma^0_k\LFP(\Sigma^0_1)$ is stronger than $\Sigma^0_1\mathchar`-\mathsf{CA}$.
  Thus, $\Sigma^0_k\LFP(\Sigma^0_1)$ is equivalent to $\mathsf{lim}$ and hence much weaker than $\ATR$ although $\Sigma^0_k\LFP(\Pi^0_1)$ is stronger than $\ATR_2$ as we saw the previous theorem.
\end{remark}

\begin{theorem}\label{thm LFP and LCP}
  Let $k,k' \in \omega, k' < k$. Then $\Delta^0_k\LFP(\Pi^0_{k'}) \leq_W \Delta^0_k\LCP(\Pi^0_{k'})$.
\end{theorem}
\begin{proof}
  Let $\theta(n,X,Y)$ be a $\Pi^0_{k'}$ $X$-positive formula and $A$ be a set.
  We will show that any output of $\Delta^0_k\LCP(\Pi^0_{k'})(\theta,A)$ is an output of $\Delta^0_k\LFP(\Pi^0_{k'})(\theta,A)$.
  In the following argument, we write $\Gamma$ to mean the monotone operator defined by $\theta$ and $A$.

  Let $X$ be an output of $\Delta^0_k\LCP(\Pi^0_{k'})(\theta,A)$.
  Then, $\Gamma(X) \subseteq X$.
  Thus, by the monotonicity of $\Gamma$, $\Gamma^2(X) \subseteq \Gamma(X)$. Hence $\Gamma(X)$ is also a closed point.
  Since $\Gamma$ is a $\Pi^0_{k'}$-operator, $\Gamma(X)$ is $\Pi^0_{k'}$ definable from $X$ and $A$.
  Consequently, $\Gamma(X)$ is a $\Delta^0_k(X\oplus A)$-definable closed point which is smaller or equal to $X$. Thus $\Gamma(X) = X$, that is, $X$ is a fixed point of $\Gamma$.
  Additionally, there is no $\Delta^0_k(X,A)$-definable fixed point strictly smaller than $X$ because a fixed point is a closed point.
  Therefore, $X$ is an output of $\Delta^0_k\LFP(\Pi^0_{k'})(\theta,A)$.
\end{proof}

For the reduction of $\LCP$ to $\LPP$, we use products of trees.
\begin{definition}
  For each $\sigma \in \omega^{<\omega}$ and $l \in \omega$, we define
  $\sigma_l$ as the sequence $\langle \sigma((l,0)),\ldots,\sigma((l,n_l)\rangle$.
  Here, $n_l$ is the maximum integer satisfying $(l,n_l) < |\sigma|$.
  By this definition, we can regard a sequence $\sigma$ as a sequence of sequences $\langle \sigma_l \rangle_{l < l'}$ for some $l'$.

  Let $\langle T_l \rangle_l$ be a sequence of ill-founded trees.
  Define a tree $\mathcal{S}(\langle T_l\rangle_l) = \{\sigma \in \omega^{<\omega} : (\forall l < |\sigma|)(\sigma_l \in T_l)\}$.
\end{definition}
\begin{lemma}\label{Suslin-op}
  There exists a computable bijection between $[\mathcal{S}(\langle T_l\rangle_l)]$ and $\prod_l[T_l] = [T_0] \times [T_1] \times \cdots$.
\end{lemma}
\begin{proof}
  For each $f \in [\mathcal{S}(\langle T_l\rangle_l)]$, define $\langle f_l \rangle_l$ by $f_l(n) = f(l,n)$.
  This correspondence is the desired one.
\end{proof}
By this correspondence, we identify a path $f$ and a sequence of paths $\langle f_l \rangle_l$.

\begin{lemma}(Essentially \cite[Lemma 2.4.]{Caristi-fixed-point})
  The following problem $\P$ is reducible to $\Delta^0_n\LPP$.
  \begin{itembox}[l]{$\P$}
  \begin{description}
    \item[\sf Input] A sequence $\langle T_l \rangle_l$ of trees.
    \item[\sf Output] A set $g$ such that $g$ computes the set $\{l : [T_l] \cap \Sigma^0_n(g,\langle T_l \rangle_l) \neq \varnothing\}$.
  \end{description}
\end{itembox}
\end{lemma}
\begin{proof}
  Let $\langle T_l \rangle_l$ be a set of trees.
  For each $l$, put $T'_l = \{(0)\ast \sigma : \sigma \in T_l\} \cup \{1^n : n \in \omega\}$.
  Then, each $T'_l$ has the rightmost path $1^\infty = \langle 1,1,1,\ldots \rangle$.

  Put $S = \mathcal{S}(\langle T'_l \rangle_l)$.
  We note that $1^\infty$ is a path of $S$. Thus, we can apply $\Delta^0_k\LPP$ to $(S,1^\infty)$.
  Let $g$ be an output of $\Delta^0_k\LPP(S,1^\infty)$.
  Then each $g_l$ is a path of $T'_l$.

  \setcounter{claimcounter}{1}
  \begin{claim}
    Put $X = \{l : g_l(0) = 0\}$. Then $X = \{l : [T_l] \cap \Sigma^0_n(g,\langle T_l \rangle_l) \neq \varnothing\}$.
  \end{claim}

  \begin{proof of claim}
    The inclusion $X \subseteq \{l : [T_l] \cap \Sigma^0_n(g,\langle T_l \rangle_l) \neq \varnothing\}$ is immediate. Thus, we will show that the other inclusion holds.

    Let $l_0$ be such that $[T_{l_0}] \cap \Sigma^0_n(g,\langle T_l \rangle_l) \neq \varnothing$ but $g_{l_0}(0) = 1$.
    Pick $h \in [T_{l_0}] \cap \Sigma^0_n(g,\langle T_l \rangle_l)$ and let $\langle g'_l \rangle_l$ be the sequence
    obtained from $\langle g_l \rangle_l$ by replacing $g_{l_0}$ with $(0)\ast h$. Then $g' = \langle g'_l \rangle_l$ is also a path of $S$ and it is left to $g$.
    Since $g'$ is $\Sigma^0_n(g,\langle T_l \rangle_l)$-definable, this contradicts the fact that $g$ is an output of  $\Delta^0_k\LPP(S,1^{\infty})$.
    Therefore, we have $g_{l_0}(0) = 0$.
    \qedclaim
  \end{proof of claim}
  Since this $X$ is computable from $g$, we conclude that $g$ satisfies the desired condition.
\end{proof}

\begin{theorem}\label{thm LCP and LPP}(Essentially \cite[Lemma 2.4.]{Caristi-fixed-point})
  Let $k \in \omega$. Then $\Delta^0_k\LCP(\Sigma^0_2) \leq_W \Delta^0_k\LPP$.
\end{theorem}
\begin{proof}
  Let $\Gamma$ be a $\Sigma^0_2$ monotone operator.
  Then the formula $\theta(x,Y) \equiv \Gamma(Y) \subseteq Y \land x \in Y^c$ is $\Pi^0_2$, and hence there exists a $\Delta^0_0$ formula $\varphi$ such that
  \begin{itemize}
    \item $(\forall x,Y)(\theta(x,Y) \leftrightarrow \exists f\in\omega^{\omega} \forall m \varphi(x,Y[m],f[m]))$,
    \item if the righthand side is true, then a witness $f$ is computable from $Y$ uniformly.
  \end{itemize}

  Pick a sequence $\langle T_x \rangle_x$ of trees such that
  \begin{align*}
    (\forall x,Y,f)(Y \oplus f \in [T_x] \leftrightarrow \forall m\varphi(x,Y[m],f[m])).
  \end{align*}
  By $\Delta^0_k\LPP$, we can find a set $g$ such that
  \begin{align*}
    X = \{x : [T_x] \cap \Delta^0_k(g,\langle T_x \rangle_x) \neq \varnothing\} \leq_T g.
  \end{align*}
  Now we have
  \begin{align*}
    x \in X &\leftrightarrow (\exists Y \oplus f \in \Delta^0_k(g,\langle T_x \rangle_x))  (\forall m)\varphi(x,Y[m],f[m]) \\
    &\to(\exists Y \in \Delta^0_k(g,\langle T_x \rangle_x)) (\exists f) (\forall m)\varphi(x,Y[m],f[m]) \\
    &\to (\exists Y \in \Delta^0_k(g,\langle T_x \rangle_x)) (\exists f \leq_T Y) (\forall m)\varphi(x,Y[m],f[m]) \\
    &\to x \in X.
  \end{align*}
  Therefore, $x \in X$ is equivalent to
  $\exists Y \in \Delta^0_k(g,\langle T_x \rangle_x)) \theta(x,Y)$. Thus we have
  \begin{align*}
    X &= \bigcup_{\substack{Y \in \Delta^0_n(g,\langle T_l \rangle_l), \\ \Gamma(Y) \subseteq Y}} Y^c,
  \end{align*}
  and hence
  \begin{align*}
    X^c &= \bigcap_{\substack{Y \in \Delta^0_n(g,\langle T_l \rangle_l), \\ \Gamma(Y) \subseteq Y}} Y.
  \end{align*}
  Since $X^c$ is an intersection of closed points, $X^c$ is also a closed point.
  We show that any $\Delta^0_k(X)$-definable closed point is bigger than or equal to $X^c$.
  Let $Y \in \Delta^0_k(X)$ be a closed point. We show that $Y^c \subseteq X$.
  Let $x \in Y^c$.
  Then, $\theta(x,Y)$ holds.
  Thus, there exists $f \leq_T Y$ such that $\forall m \varphi(x,Y[m],f[m])$.
  For such an $f$, $Y \oplus f \in [T_x] \cap \Delta^0_k(g,\langle T_x \rangle_x)$.
  Therefore, $x \in X$.
\end{proof}

\begin{corollary}
  For any $k$, we have
  \begin{align*}
    \Delta^0_k\LCP(\Pi^0_1) \equiv_W \Delta^0_k\LCP(\Sigma^0_2) \equiv_W \Delta^0_k\LPP.
  \end{align*}
  Moreover, if $k = 2$, then
  \begin{align*}
    \Delta^0_k\LPP \equiv_W \Delta^0_k\LFP(\Pi^0_1),
  \end{align*}
  and if $k \geq 3$, then
  \begin{align*}
    \Delta^0_k\LPP \equiv_W \Delta^0_k\LFP(\Pi^0_1) \equiv_W \Delta^0_k\LFP(\Sigma^0_2).
  \end{align*}
\end{corollary}
\begin{proof}
  Let $k \in \omega$.
  Then,
  \begin{itemize}
    \item $\Delta^0_k\LCP(\Pi^0_1) \leq_W \Delta^0_k\LCP(\Sigma^0_2)$ trivially holds,
    \item $\Delta^0_k\LCP(\Sigma^0_2) \leq_W \Delta^0_k\LPP$ follows from Theorem \ref{thm LCP and LPP} and
    \item $\Delta^0_k\LPP \leq_W \Delta^0_k\LCP(\Pi^0_1)$ follows from Remark \ref{rmk LPP and LCP}.
  \end{itemize}
  Therefore, we have
  \begin{align*}
    \Delta^0_k\LCP(\Pi^0_1) \equiv_W \Delta^0_k\LCP(\Sigma^0_2) \equiv_W \Delta^0_k\LPP.
  \end{align*}
  Assume $k =2$. Then,
  \begin{itemize}
    \item $\Delta^0_k\LPP \leq_W \Delta^0_k\LFP(\Pi^0_1)$ follows from Theorem \ref{from LFP to LPP} and
    \item $\Delta^0_k\LFP(\Pi^0_1) \leq_W \Delta^0_k\LCP(\Pi^0_1)$ follows from Theorem \ref{thm LFP and LCP}.
  \end{itemize}
  As we have proved just before, $\Delta^0_k\LCP(\Pi^0_1) \equiv_W \Delta^0_k\LPP$. Therefore, we have
  \begin{align*}
    \Delta^0_k\LPP \equiv_W \Delta^0_k\LFP(\Pi^0_1).
  \end{align*}
  Assume $k \geq 3$.
  The equivalence $\Delta^0_k\LPP \equiv_W \Delta^0_k\LFP(\Pi^0_1)$ is proved by the same way as in the case $k =2$. In addition, we have already proved $\Delta^0_k\LPP \equiv_W \Delta^0_k\LCP(\Sigma^0_2)$.
  We note that $\Delta^0_k\LFP(\Sigma^0_2) \leq_W \Delta^0_k\LCP(\Sigma^0_2)$ is proved by the same way as in the proof of Theorem \ref{thm LFP and LCP}.
  Therefore, we have
  \begin{align*}
    \Delta^0_k\LPP \leq_W \Delta^0_k\LFP(\Pi^0_1)\leq_W \Delta^0_k\LFP(\Sigma^0_2) \leq_W \Delta^0_k\LCP(\Sigma^0_2) \equiv_W \Delta^0_k\LPP
  \end{align*}
  which implies $\Delta^0_k\LPP \equiv_W \Delta^0_k\LFP(\Pi^0_1) \equiv_W\Delta^0_k\LFP(\Sigma^0_2)$.
\end{proof}

\section{Questions}
In Section 3, we studied $\FP,\ATR_2$ and their relationship.
We proved that
\begin{enumerate}
  \item $\ATR_2 <_W \FP(\Sigma^0_2)$.
  \item $\ATR_2 \times \ATR_2 \not \leq_W \ATR_2$, and hence $\ATR_2$ is not parallelizable.
\end{enumerate}

\begin{question}
  Open questions in Section 3:
  \begin{enumerate}
    \item Is $\FP(\Sigma^0_2) \leq_W^a \ATR_2$?
    \item What is the relationship between $\FP(\Sigma^0_n)$ and $\FP(\Sigma^0_m)$?
    \item Let $(\ATR_2)^n$ denote the $n$-times parallel product of $\ATR_2$.
    What is the relationship between $(\ATR_2)^n$ and $(\ATR_2)^m$?
  \end{enumerate}
\end{question}

In Section 4, we introduced the $\omega$-model reflection of Weihrauch problems and studied its properties.
We proved that
\begin{itemize}
  \item if $\P$ is a $(\Pi^1_1,\Sigma^1_0)$-representable problem, then $\P \leq_W \P^{\rfn}$,
  \item if $\P$ is an arithmetically representable problem, then $\P <_W \P^{\rfn}$.
\end{itemize}

\begin{question}
  Is there a $(\Pi^1_1,\Sigma^1_0)$-representable problem $\P$ such that $\P \equiv \P^{\rfn}$?
\end{question}

\subsection*{Acknowledgement}
The first author's work is supported by JST, the establishment of university
fellowships towards the creation of science technology innovation,
Grant Number JPMJFS2102.
The second author's work is partially supported by JSPS KAKENHI grant numbers JP19K03601, JP21KK0045 and JP23K03193.
We would like to thank Leonardo Pacheco and Tadayuki Honda for their comments on this paper.
We also would like to thank the anonymous reviewer for his or her fruitful comments.

\bibliography{references}
\bibliographystyle{plain}

\end{document}